\documentclass{amsart}

\usepackage[small, nohug]{diagrams}
\usepackage[mathscr]{eucal}
\usepackage{amsfonts}
\usepackage{amsmath}
\usepackage{amsthm}
\usepackage{amssymb}
\usepackage{latexsym}

\usepackage{graphicx}
\usepackage{color}
\usepackage[noadjust]{cite}
\usepackage{epsfig}

\newfont{\msam}{msam10}
\newtheorem{theorem}{Theorem}[section]
\newtheorem{proposition}[theorem]{Proposition}
\newtheorem{corollary}[theorem]{Corollary}
\newtheorem{lemma}[theorem]{Lemma}
\theoremstyle{definition}
\newtheorem{definition}[theorem]{Definition}
\newtheorem{remark}[theorem]{Remark}
\newtheorem{rmk}[theorem]{Remark}
\newtheorem{example}[theorem]{Example}

\newtheorem{question}[theorem]{Question} 

\let\nc\newcommand

\nc{\la}{\label}

\def\bthm{\begin{theorem}}
\def\ethm{\end{theorem}}
\def\blemma{\begin{lemma}}
\def\elemma{\end{lemma}}
\def\bproof{\begin{proof}}
\def\eproof{\end{proof}}
\def\bprop{\begin{proposition}}
\def\eprop{\end{proposition}}

\def\Z{\mathbb{Z}}

\def\A{\SH_{q,t}}
\def\SH{\mathrm{SH}}
\def\H{\mathrm{H}}
\def\HH{\mathrm{H}}

\def\N{\mathbb{N}}

\def\U{\mathcal{U}}
\def\D{\mathcal{D}}

\def\s{\hat{s}}
\def\sx{\hat{x}}

\def\sy{\hat{y}}
\def\sT{\hat{T}}
\def\e{\boldsymbol{\mathrm{e}}}

\def\g{\mathfrak{g}}
\def\tc{{t_c}}

\def\sl{\mathfrak{sl}}

\def\c{\mathbb{C}}
\def\C{\mathbb{C}}

\nc{\Hom}{{\rm{Hom}}}
\nc{\Ext}{{\rm{Ext}}}
\nc{\htau}{{\bar{ t}}}
\nc{\HOM}{\underline{\rm{Hom}}}
\nc{\EXT}{\underline{\rm{Ext}}}
\nc{\TOR}{\underline{\rm{Tor}}}
\nc{\End}{{\rm{End}}}
\nc{\Map}{{\rm{Map}}}
\nc{\Out}{{\rm{Out}}}
\nc{\GL}{{\rm{GL}}}
\nc{\SL}{{\rm{SL}}}
\nc{\PGL}{{\rm{PGL}}}
\nc{\G}{{\rm{G}}}
\nc{\Rep}{{\rm{Rep}}}
\nc{\ad}{{\rm{ad}}}
\nc{\dlim}{\varinjlim}
\newcommand{\Mod}{{\tt{Mod}}}

\newcommand{\into}{\,\,\hookrightarrow\,\,}

\numberwithin{equation}{section}

\newcommand{\inv}[1]{#1^{-1}}

\newcommand{\rr}{{\mathbf{r}}}
\newcommand{\sss}{{\mathbf{s}}}

\newcommand{\aaa}{{\mathbf{a}}}
\newcommand{\tp}{ {\mathrm{top} } }
\newcommand{\alg}{{\mathrm{alg}}}
\newcommand{\gl}{\mathfrak{gl}}

\begin{document}
\title[]{Iterated torus knots and double affine Hecke algebras}
\date{\today}
\author{Peter Samuelson}
\address{Department of Mathematics, University of Edinburgh, Edinburgh, UK}
\email{peter.samuelson@ed.ac.uk}

\begin{abstract}
We give a topological realization of the (spherical) double affine Hecke algebra $\SH_{q,t}$ of type $\sl_2$, and we use this to construct a module over $\SH_{q,t}$ for any knot $K \subset S^3$. As an application, we give a purely topological interpretation of Cherednik's 2-variable polynomials $P_n(r,s; q,t)$ of type $\sl_2$ from \cite{Che13} (where $r,s \in \Z$ are relatively prime). 

We then generalize the construction of these polynomials (for $\sl_2$) from torus knots to all iterated cables of the unknot and prove they specialize to the colored Jones polynomials of the knot. Finally, in the Appendix we compare this construction to a later construction of Cherednik and Danilenko.
\end{abstract}
\maketitle
\setcounter{tocdepth}{1}
\tableofcontents

\section{Introduction}
The (spherical) \emph{double affine Hecke algebra} $\SH_{q,t}(\g)$ is a (noncommutative) algebra associated to a semisimple complex Lie algebra $\g$ and two parameters $q,t \in \C^*$. These algebras were introduced by Cherednik to prove Macdonald's conjectures about Macdonald polynomials \cite{Che95}, and they have since found applications in many areas of mathematics (see, e.g. \cite{Che05} and references therein). In this paper we give two new connections between DAHAs and knots -- the first involves iterated cables of the unknot, and the second involves arbitrary knots.

\subsection{Cables}
In \cite{AS15} (see also \cite{AS12} and \cite{AS11}), Aganagic and Shakirov used refined Chern-Simons theory and generalized Verlinde algebras to construct $q,t$ versions of Reshetikhin-Turaev invariants for torus knots. In \cite{Che13}, Cherednik gave a construction of these polynomials using double affine Hecke algebras. More precisely, he used the representation theory of $\SH_{q,t}(\g)$ to construct polynomials $P_{\alpha,r,s}(q,t) \in \C[q^{\pm 1},t^{\pm 1}] $, where $\alpha \in \mathfrak t^*$ is an integral dominant weight and $r,s \in \Z$ are relatively prime. He gave a number of conjectural properties of these polynomials -- the conjecture which is relevant to us is that they specialize at $t=q$ to the Reshetikhin-Turaev invariants of the $(r,s)$ torus knot.
 
From now on we fix $\g = \sl_2$. In this case, if we identify integral dominant weights with $\N$,  Cherednik proved the equality\footnote{This equality is stated in our normalization conventions - the precise conversion is stated in Remark \ref{rmk_normalization}.}
\begin{equation}\label{eq_specialization_intro}
P_{n,r,s}(q,-q^2) = J_n(K_{r,s};q)
\end{equation}
where the right hand side is the $n^{\textrm{th}}$ colored Jones polynomial of $K_{r,s}$. 

In this paper we extend the construction of the polynomials $P_{n,r,s}(q,t)$ from torus knots to all \emph{iterated cables} of the unknot and show that these new polynomials specialize to the colored Jones polynomials. To do this, we provide a \emph{cabling formula} that expresses the colored Jones polynomials of the cable $K_{r,s}$ of a knot $K\subset S^3$ in terms of those of $K$. Various versions of this formula are well-known and have appeared in the literature, but for completeness we prove skein-theoretic versions which are suited for our purposes in Section \ref{sec_cabling}. We then show that the $t=-q^2$ specialization of the formula defining the polynomial $P_{n}(r,s;q,t)$ is identical to the cabling formula when $K$ is the unknot (in this case the cable $K_{r,s}$ is the $(r,s)$ torus knot.)

We remark that the formula defining the polynomials $P_n(r,s;q,t)$ (and our formula generalizing this to iterated torus knots) uses several structures associated to the DAHA $\SH_{q,t}$, and when this formula is identified with the skein-theoretic cabling formula, each of these structures has a natural topological interpretation. More precisely, Cherednik's construction uses the algebra $\SH_{q,t}$, an $\SL_2(\Z)$ action on $\SH_{q,t}$, an action of $\SH_{q,t}$ on $P = \C[x]$ (called the polynomial representation), and a pairing $P^{op}\otimes_{\SH_{q,t}} P \to \C[q^{\pm 1},t^{\pm 1}]$ (where $P^{op}$ is the twist of $P$ by a certain anti-automorphism). In the $t=-q^2$ specialization the algebra $\SH_{q,-q^2}$ is the Kauffman bracket skein algebra of the torus, and the $\SL_2(\Z)$ action on $\SH_{q,t}$ is induced by the action of the mapping class group of the torus. The polynomial representation corresponds to the skein module of a small neighborhood of a knot, which is a module over the algebra associated to its boundary torus. Finally, Cherednik's pairing corresponds to the decomposition of $S^3$ into the union of two unknotted solid tori, and the element on which he evaluates the pairing corresponds to the $(r,s)$ torus knot, embedded in the common boundary of the two solid tori and colored with a Jones-Wenzl idempotent.

One advantage of using the cabling formula is that it applies to all knots. As mentioned above, this allows us to extend the definition of the polynomial $P_{n,r,s}$ from torus knots to all iterated cables of the unknot. 
More precisely, given a sequence $ \rr_m = (r_1,\ldots,r_m)$ and $\sss_m = (s_1,\ldots,s_m)$ with $r_i,s_i$ relatively prime integers, we define $K(\rr_1,\sss_1)$ to be the $(r_1,s_1)$-torus knot and $K(\rr_m,\sss_m)$ to be the $(r_m,s_m)$ cable of $K(\rr_{m-1},\sss_{m-1})$. We then define a rational function $P_n(\rr_m,\sss_m; q,t) \in \C(q^{\pm 1},t^{\pm 1})$ for each $n \in \N$ (see Definition \ref{def_iteratedcable}) and prove the following (see Theorem \ref{thm_topiteratedcablespecialization}):
\begin{theorem}
If $K(\rr,\sss)$ is a knot which is an iterated cable of the unknot, we have the equality 
\[
 P_n(\rr,\sss;q,t=-q^2) = J_n(K(\rr,\sss); q)
\]
\end{theorem}
If the sequences $\rr_m$ and $\sss_m$ are length 1, then $K(\rr,\sss)$ is a torus knot and this theorem reproduces one of Cherednik's theorems in \cite{Che11}.

\subsection{Arbitrary knots}
It is natural to ask whether some similar construction can be used to produce 2-variable polynomial knot invariants for \emph{all} knots. 
However, it seems likely that algebraic constructions using the polynomial representation can only ``see'' iterated torus knots or links. 

As a first step in this direction, in Section \ref{sec_knotcomplementsmodule}, for \emph{any} knot $K$ we define an $\SH_{q,t}$ bimodule $\bar K_{q,t}$ using a $t$-modification of the Kauffman bracket skein module construction. For the unknot, we show this bimodule is isomorphic to $\SH_{q,t}$ itself. We further construct a canonical quotient $K_{q,t}$ of $\bar K_{q,t}$ which is just a left module (see Definition \ref{def_leftmodquotient}), and we show that for the unknot this left module is the sign representation of $\SH_{q,t}$. We also show the following (see Proposition \ref{prop_teq1surj}):
\begin{theorem}
The vector space $K_{q,t}$ is knot invariant which is a left module over $\SH_{q,t}$. If we specialize $t=1$ then there is a natural surjective $\SH_{q,t=1}$-module map
\begin{equation}\label{eq_introsurj} 
K_{q,t=1}  \twoheadrightarrow K_q(S^3\setminus K).
\end{equation}
\end{theorem}
We remark that there have been a number of recent papers providing (or conjecturing) various connections between double affine Hecke algebras and certain classes of knots (e.g.\ \cite{Che13}, \cite{GORS14}, \cite{EGL13}, \cite{GN15}, \cite{BS14}). However, to the best of our knowledge, the $\SH_{q,t}$-module $K_{q,t}$ is the only proven connection between DAHAs (with arbitrary parameters) and arbitrary knots.

To extract polynomial knot invariants from the classical skein module, one uses the $\C[q^{\pm 1}]$-linear map \[\epsilon: K_q(S^3\setminus K) \to K_q(S^3) = \C[q^{\pm 1}]\] 
induced by the inclusion $S^3\setminus K \to S^3$ - the key fact here is that $K_q(S^3)$ is isomorphic to $\c[q^{\pm 1}]$. 
\begin{question}
 For all $t \in \C^*$, is there a canonical evaluation map $\bar K_{q,t} \to \C[q^{\pm 1},t^{\pm 1}]$?
\end{question}
We give a positive answer to this question for the unknot in Corollary \ref{corollary_pkp}, but for other knots this seems to be a subtle question. In particular, the proof of Corollary \ref{corollary_pkp} uses the PBW property for $\H_{q,t}$, which is a nontrivial fact. It is unclear whether an analogue of this PBW property can be proven for nontrivial knots. However, composing the surjection (\ref{eq_introsurj}) with the evaluation map $\epsilon$ gives a positive answer to this question when $t=1$.
 

One might also ask if there is a similar topological construction of the (spherical) DAHA if $\g$ has rank greater than 1. Two versions of skein relations for $\g = \sl_n$ (and $t=1$) appear in \cite{Sik05} and \cite{CKM14}, and skein relations for $\g$ of rank 2 appear in \cite{Kup96}. However, it is not clear whether the spherical subalgebra $\SH_{q,t}(\g)$ is a quotient of the $\g$-skein module of the punctured torus, and the ``generators-and-relations" approach in this paper will be more difficult for higher rank $\g$.

Some brief historical remarks are in order. After the first version of the present paper appeared, Cherednik and Danilenko (see \cite{CD14}) gave a construction of certain polynomials for general $\g$ that are conjecturally related to iterated torus knots. For $\g = \sl_2$ we compare their construction to ours in an Appendix. After this, Morton and the author posted \cite{MS14} (now \cite{MS17}), which proved that the $\gl_n$ polynomials defined in \cite{CD14} specialized to the $\gl_n$ Reshetikhin-Turaev invariants for iterated torus knots. Without further argument, this theorem doesn't directly imply the statement for $\sl_2$ polynomials proved in this paper because the relationship between the $q,t$ polynomials for $\gl_2$ and $\sl_2$ is not obvious.

A summary of the contents of the paper is as follows. In Section \ref{sec_background} we recall background material about the Kauffman bracket skein module and the double affine Hecke algebra, including two cabling formulas that are essential in later sections. In Section \ref{sec_topologicaldeformations} we construct $t$-deformed versions of the Kauffman bracket skein module of a surface and of knot complements. We then use this to give a topological construction of Cherednik's polynomials in Section \ref{sec_topcherednik}, and we give a topological proof that these polynomials specialize to the colored Jones polynomials of torus knots. We give an algebraic generalization of Cherednik's construction to iterated cables of the unknot in Section \ref{sec_iteratedcables}. Finally, in an Appendix we show that when $\g = \sl_2$, the polynomials defined in \cite{CD14} for iterated torus knots specialize to some constant times the colored Jones polynomial.

\noindent \textbf{Acknowledgements:}
We would like to thank Yuri Berest for extensive explanations and discussions while advising the author's thesis \cite{Sam12} (which contained the results in the first three subsections of Section \ref{sec_topologicaldeformations}), and for providing notes on which Section \ref{sec_dahabackground} was based. We would also like to thank I. Cherednik for helpful comments on an early version of this paper and discussions of \cite{CD14}, and G. Masbaum for several conversations, and in particular for explaining the sign in Corollary  \ref{lemma_topskeincable}. We also thank D. Bar-Natan, E. Gorsky, A. Marshall, G. Muller, A. Oblomkov, M. Pabiniak, S. Shakirov, and D. Thurston for enlightening conversations. The author is grateful to the users of the website MathOverflow who have provided several helpful answers (see, e.g. \cite{35687}). Finally, several computations were done using \texttt{Mathematica}, and in particular using the packages \texttt{KnotTheory} and \texttt{NCAlgebra}.

\section{Background}\label{sec_background}
In this section we recall background material about Kauffman bracket skein modules and double affine Hecke algebras that will be used in the remainder of the paper.

\subsection{Kauffman bracket skein modules}
In this section we define the Kauffman bracket skein module and recall several of its properties. In particular, we describe its relation to the colored Jones polynomials and two resulting cabling formulas.

\subsubsection{Knot complements}\label{subsec_knotgroups}


Recall that two maps $f,g:M \to N$ of manifolds are \emph{ambiently isotopic} if they are in the same orbit of the identity component of the diffeomorphism group of $N$. This is an equivalence relation, and a \emph{knot} in a 3-manifold $M$ is the equivalence class of a smooth embedding $K: S^1 \hookrightarrow M$. 

For an oriented knot $K \subset S^3$ there is a canonical identification $T = S^1 \times S^1 \to \partial(S^3 \setminus K)$. More precisely, let  $N_K \subset S^3$ be a closed tubular neighborhood of $K$, and let $N_c$ be the closure of its complement. Then the following lemma provides a unique (up to isotopy) identification of $N_K \cap N_c$ with $S^1\times S^1$:
\begin{lemma}[{\cite[Ch. 3]{BZ03}}]\label{lemma_meridianlongitude}
 There is a unique (up to isotopy) pair of simple loops  (the meridian $m$ and longitude $l$)  in $T$ subject to the conditions
\begin{enumerate}
\item  $m$ is nullhomotopic in $N_K$,
\item $l$ is nullhomologous in $N_c$,
\item $m,l$ intersect once in $T$,
\item in $S^3$, the linking numbers $(m,K)$ and $(l,K)$ are 1 and 0, respectively.
\end{enumerate}
\end{lemma}

\subsubsection{Kauffman bracket skein modules}\label{kbsmsection}
A \emph{framed link} is an 
embedding of a disjoint union of annuli $S^1 \times [0,1]$ into an oriented 3-manifold $M$. (The framing refers to the $[0,1]$ factor and is a technical detail that will be suppressed when possible.) We will consider framed links to be equivalent if they are ambiently isotopic.

Let  $\mathscr L(M)$ be the vector space 
spanned by the set of ambient isotopy classes of framed unoriented links in $M$ (including the empty link). Let $\mathscr L'(M)$ be the smallest subspace of $\mathscr L(M)$ containing the skein expressions 
$L_+ - qL_0 - q^{-1}L_\infty$ and $L \sqcup \bigcirc + (q^2+q^{-2})L$. The links $L_+$, $L_0$, and $L_\infty$ are identical outside of a small 3-ball (embedded as an oriented sub-manifold of $M$), and inside the 3-ball 
they appear as in Figure \ref{kbsm}. (All pictures drawn in this paper will have blackboard framing. In other words, a line on the page represents a strip $[0,1]\times [0,1]$ in a tubular neighborhood of the page, 
and the strip is always perpendicular to the paper.)

\begin{remark}
Our constant $q$ is the same as the constant $A$ that is more commonly used in skein theory (e.g. in \cite{BHMV95}). (We made this notational choice since we reserve the notation $A$ for an algebra.)
\end{remark}

\begin{figure}
\begin{center}
\input{kbsmrelation.pstex_t}
\caption{Kauffman bracket skein relations}\label{kbsm}
\end{center}
\end{figure}

\begin{definition}[\cite{Prz91}]
The \textbf{Kauffman bracket skein module} is the vector space $K_q(M) := \mathscr L / \mathscr L'$. It contains a canonical element $\varnothing \in K_q(M)$ corresponding to the empty link.
\end{definition}

\begin{remark}
 To shorten the notation, if $M = F \times [0,1]$ for a surface $F$, we will often write $K_q(F)$ for the skein module $K_q(F\times [0,1])$.
\end{remark}

\begin{example}\label{skeins3}
One original motivation for defining $K_q(M)$ is the isomorphism of vector spaces
\[\C \stackrel \sim \to K_q(S^3), \quad 1 \mapsto \varnothing \]
Kauffman proved that this map is an isomorphism and that the inverse image of a link is the Jones polynomial\footnote{More precisely, the image of the link is a number in $\C$ that depends polynomially on $q \in \C^*$, and this (Laurent) polynomial is the Jones polynomial of the link, up to a normalization.} of the link. The skein relations in Figure \ref{kbsm} can be used to remove crossings and trivial loops of a diagram of a link until the diagram is a multiple of the empty link, which shows that the vector space map $\C \to K_q(S^3)$ sending $\alpha \mapsto \alpha\cdot \varnothing$ is surjective. Showing it is injective is equivalent to showing the Jones polynomial of a link is well-defined.
\end{example}

In general $K_q(M)$ is just a vector space - however, if $M$ has extra structure, then $K_q(M)$ also has extra structure. In particular,
\begin{enumerate}
\item If $M = F \times [0,1]$ for some surface $F$, then $K_q(M)$ is an algebra (which is typically noncommutative). The multiplication is given by ``stacking links." 
\item If $M$ is a manifold with boundary, then $K_q(M)$ is a module over $K_q(\partial M)$. The action is given by ``pushing links from the boundary into the manifold.'' 
\item\label{embedding} An oriented embedding $M \hookrightarrow N$ of 3-manifolds induces a linear map $K_q(M) \to K_q(N)$. 
\item If $q=\pm 1$, then $K_q(M)$ is a commutative algebra (for any oriented 3-manifold $M$). The multiplication is given by ``disjoint union of links,'' which is well-defined because when $q=\pm 1$, the skein relations allow strands to `pass through' each other.
\end{enumerate}

\begin{remark}
The third property may be interpreted as follows: let $C$ be the category whose objects are oriented 3-dimensional manifolds and whose morphisms are oriented embeddings. Then $K_q(-)$ is a functor\footnote{To be pedantic, $K_q(-)$ is functorial with respect to maps $M \to N$ that are oriented embeddings when restricted to the interior of $M$. In particular, if we identify a surface $F$ with a boundary component of $M$ and $N$, then the gluing map $M \sqcup N \to M \sqcup_F N$ induces a linear map $K_q(M) \otimes_\C K_q(N) \to K_q(M\sqcup_F N)$.} from $C$ to the category of vector spaces. We also remark that the first two properties are a special case of the third. For example, there is an obvious map $F\times [0,1] \sqcup F\times [0,1] \to F\times [0,1]$, and the product structure of $K_q(F \times [0,1])$ comes from the application of the functor $K_q(-)$ to this map.
\end{remark}

\begin{example}
 Let $M = (S^1 \times [0,1]) \times [0,1]$ be the solid torus. The skein relations can be applied to remove crossings and trivial loops in a diagram of any link, and the result is a sum of unions of parallel copies of the loop $u$ generating $\pi_1(M)$. This shows that algebra map $\C[u] \to K_q(S^1\times[0,1])$ sending $u^n$ to $n$ parallel copies of $u$ is surjective, and it follows from \cite{SW07} that this map is injective.
\end{example}

%
%
%
%

\subsubsection{The Kauffman bracket skein module of the torus}
We recall that the \emph{quantum torus} is the algebra
\[
A_q := \frac{\C\langle X^{\pm 1},Y^{\pm 1}\rangle}{XY-q^2YX}
\]
where $q \in \C^*$ is a parameter. There is a $\Z_2$ action by algebra automorphisms on $A_q$ where the generator simultaneously inverts $X$ and $Y$. We define $e_{r,s} = q^{-rs}X^{r}Y^s \in A_q$, which form a linear basis for the quantum torus $A_q$ and satisfy the relations
\[e_{r,s}e_{u,v} = q^{rv-us}e_{r+u,s+v}\]

In this section we recall a beautiful theorem of Frohman and Gelca in \cite{FG00} that gives a connection between skein modules and the invariant subalgebra $A_q^{\Z_2}$. First we establish some notation. Let $T_n \in \c[x]$ be the Chebyshev polynomials defined 
by $T_0 = 2$, $T_1 = x$, and the relation $T_{n+1} = xT_n-T_{n-1}$. If $m,l$ are relatively prime, write $(m,l)$ for the $m,l$ curve on the torus (which is the simple curve wrapping around 
the torus $l$ times in the longitudinal direction and $m$ times in the meridian's direction). It is clear that the links $(m,l)^n$ span $K_q(T^2)$, and it follows from \cite{SW07} that this set is a basis.  However, a more convenient basis is given by the elements $(m,l)_T = T_d((\frac m {d}, \frac l {d}))$ (where $d = \mathrm{gcd}(m,l)$). 
 (We point out that since we are considering unoriented curves, $(m,l) = (-m,-l)$.)

\begin{theorem}[\cite{FG00}]\label{fg00}
The map $f:K_q(T^2) \to A_q^{\Z_2}$ given by $f((m,l)_T) = e_{m,l}+e_{-m,-l}$ is an isomorphism of algebras.
\end{theorem}
(This theorem also follows from combining the results of \cite{BP00} with Theorem \ref{sphericalrelations}.)

\begin{remark}\label{remark_canonicalmodulestructure}
 From the discussion in Section \ref{subsec_knotgroups}, if $K$ is an oriented knot, then there is a canonical identification of $S^1\times S^1$ with the boundary of $S^3\setminus K$, so $K_q(S^3\setminus K)$ has a canonical $A_q^{\Z_2}$-module structure. (In fact, this module structure does not depend on the orientation of $K$, but we do not need this fact.)
\end{remark}

\subsubsection{A topological pairing and the colored Jones polynomials} \label{subsec_coloredJpolys}
Let $K \subset S^3$ be a knot. If we identify the solid torus $D^2\times S^1$ with a neighborhood of $K$, then the embeddings $D^2\times S^1 \hookrightarrow  S^3$ and $S^3\setminus K\hookrightarrow S^3$ induce a $\C$-linear map
\begin{equation}\label{eq_prepairing} 
K_q(D^2\times S^1)\otimes_\C K_q(S^3\setminus K) \to \C 
\end{equation}
If $\alpha$ is a link which is parallel to the boundary of $D^2\times S^1$, it can be isotoped to a link inside $D^2\times S^1$ or a link inside $S^3\setminus K$, and inside $S^3$ both these links are isotopic. Therefore, the map (\ref{eq_prepairing}) descends to a pairing
\begin{equation}\label{knotpairing}
\langle -,-\rangle:K_q(D^2\times S^1)\otimes_{K_q(T^2)} K_q(S^3\setminus K) \to K_q(S^3) = \C
\end{equation}

The colored Jones polynomials $J_{n}(K;q) \in \C[q^{\pm 1}]$ of a knot $K \subset S^3$ were originally defined by Reshetikhin and Turaev in \cite{RT90} using the representation theory of $\U_q(\sl_2)$. (In fact, their definition works for any semisimple Lie algebra $\mathfrak g$, but we only deal with $\mathfrak g = \sl_2$.) Here we recall a theorem of Kirby and Melvin that shows that $J_{n}(K;q)$ can be computed in terms of the pairing (\ref{knotpairing}).

Let $S_n \in \C[u]$ be the Chebyshev polynomials of the second kind, which satisfy the initial conditions $S_0 = 1$ and $S_1 = u$, and the recursion relation $S_{n+1} = uS_n - S_{n-1}$. 

\begin{theorem}[\cite{KM91}]\label{thm_coloredjonespolys}
 If $\varnothing \in K_q(S^3\setminus K)$ is the empty link, and $y \in K_q(T^2)$ is the longitude, then
 \[
  J_{n}(K;q) = \langle \varnothing \cdot S_{n-1}(y), \varnothing \rangle
 \]
\end{theorem}
\begin{remark}\label{remark_signconvention}
We remark that we avoid a common sign correction - in particular, for us, $J_{n}(\mathrm{unknot};q) = (-1)^{n-1}(q^{2n}-q^{-2n})/(q^2-q^{-2})$. Also, with this normalization, $J_{0}(K;q) = 0$ and $J_{1}(K;q) = 1$ for every knot $K$. Up to a sign, these conventions agree with the convention of labelling irreducible representations of $\U_q(\sl_2)$ by their dimension. (The $S_{n-1}$ are the characters of these irreducible representations.)
\end{remark}


\subsubsection{Two cabling formulas for colored Jones polynomials}\label{sec_cabling}
In this section we describe how the colored Jones polynomials of a cable of a knot $K$ can be computed from the skein module $K_q(S^3\setminus K)$. There are two standard ways of defining the $(r,s)$ cable of a knot, which we refer to as the \emph{topological} and \emph{algebraic} cablings, and each has its own cabling formula.

\begin{remark}
We make no claim to originality for the cabling formulas we give here - they are well known and have appeared in numerous places, including \cite[Theorem 7.1]{FG00}, \cite{vdV08}, \cite{Mor95}, and \cite{Tra14}. For the sake of completeness and self-containment we will provide precise statements.
\end{remark}

\paragraph{Topological cabling}
\begin{definition}\label{def_topcable}
Let $K \subset S^3$ be a framed knot with 0 framing and let $r,s \in \Z$ be relatively prime. Identify $S^1 \times S^1$ with the boundary of a neighborhood of $K$ such that the first copy of $S^1$ is a meridian of $K$ and the second is the longitude determined by the framing. (We note that since $K$ has 0 framing, its longitude is the same as the (unique) longitude given by Lemma \ref{lemma_meridianlongitude}). Then the $(r,s)$ \emph{topological cable} $K^\tp_{r,s}$ of $K$ is the knot which is the image of the $(r,s)$-curve on $T$ and which is given the 0 framing.   
\end{definition}

Let $\gamma_{r,s} \in \SL_2(\Z)$ satisfy $\gamma(0,1) = (r,s)$. The mapping class group of the torus $T$ is $\SL_2(\Z)$, and this induces an action of $\SL_2(\Z)$ on $K_q(T)$. By construction, we have $\gamma_{r,s}(y) = (r,s)_T$, where $(r,s)_T$ is the $(r,s)$ curve on the torus $T$ and $y$ is the longitude of $K$. We remark that the framing of $(r,s)_T$ is parallel to the torus $T$, and in particular the knot $(r,s)_T$ is not 0-framed in $S^3$.

\begin{definition}\label{def_topneighborhoods}
To give a precise statement, we give names to neighborhoods of $K$, its boundary torus, and its cable.
\begin{enumerate}
\item Let $N_{r,s}$ be a neighborhood of $K^\tp_{r,s}$, and identify $K_q(N_{r,s}) \cong \C[u]$ (as algebras) by setting the generator $u$ to be equal to the 0-framed knot $K^\tp_{r,s}$. 
\item Let $N_T$ be a neighborhood of $T$. Identify $A_q^{\Z_2}$ with $K_q(N_T)$ by identifying $x = X+X^{-1}$ with the meridian and $y = Y+Y^{-1}$ with the topological longitude of Lemma \ref{lemma_meridianlongitude}. Since $K$ is 0-framed, the longitude $y$ is the same as the longitude determined by the framing of $K$.
\item Let $N_K$ be a neighborhood of $K$, and identify $K_q(N_K) \cong \C[y]$ (as algebras) by equating $y$ with the 0-framed knot $K$. This is notationally consistent because the knot $K$ is framed parallel to $T$, so it is the element $Y+Y^{-1} \in A_q^{\Z_2}$ that we always call $y$.
\end{enumerate}
\end{definition}
The following tautological inclusions are illustrated in Figure \ref{fig_cablepic}: 
\begin{equation}\label{eq_iotatop}
N_{r,s} \stackrel{\iota}{\hookrightarrow} N_T \stackrel \mu \hookrightarrow N_K
\end{equation}
We will write 
\begin{equation}\label{def_gammatop}
\Gamma^\tp_{r,s} := \mu \circ \iota
\end{equation}
for the composition of these inclusions. By functoriality of skein modules and the identifications above, the map $\Gamma^\tp_{r,s}$ induces a $\c$-linear map $\Gamma^\tp_{r,s}:\C[u] \to \C[y]$.

\begin{figure}
\begin{center}
\input{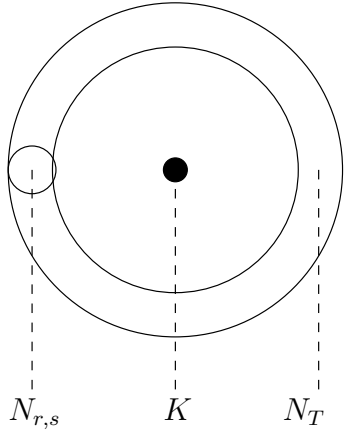}
\caption{A cross-section of the neighborhoods $N_{r,s}$ and $N_T$ for $(r,s) = (0,1)$. The outer circle bounds $N_K$. }\label{fig_cablepic}
\end{center}
\end{figure}

\begin{lemma}\label{lemma_topskeincable}
The induced $\C$-linear map $\Gamma^\tp_{r,s}: \C[u] \to \C[y]$ is given by 
\[
\Gamma^\tp_{r,s}(S_{n-1}(u)) = (-q)^{rs(n^2-1)}1_y\cdot \gamma_{r,s}(S_{n-1}(y))
\]
where $1_y \in \C[y]$  is the empty link in the skein module of a neighborhood of $K$.
\end{lemma}
\begin{proof}
We first compute the image $\iota(u)$. By definition, $\iota(u)$ is the $r,s$ curve on $T$, which is given the $0$-framing in $S^3$. The element $\gamma_{r,s}(y) \in K_q(T)$ is, by definition, the same curve, but its framing is parallel to $T$. Now, the first sentence of the second paragraph of \cite[pg. 323]{MM08} says that since $K$ has writhe 0 (i.e. framing 0), that the framing of $\gamma_{r,s}(y)$ in $S^3$ is $rs$. Therefore, $\iota(u)$ should be twisted by $rs$ units of framing to be isotopic to $\gamma_{r,s}(y)$. Now $S_{n-1}(u)$ is equal to the insertion of the $n^\textrm{th}$ Jones-Wenzl idempotent on $u$, which means that $\iota(S_{n-1}(u))$ is equal to the $n^{\textrm{th}}$ Jones-Wenzl idempotent inserted on the framed curve $\gamma_{r,s}(y)$, which is then twisted by $rs$ full twists. 

Since the Jones-Wenzl idempotent annihilates cups and caps, the $rs$ full twists simplify (under the skein relations) to the displayed power of $-q$. (For a reference for this final statement, see the first equation of the first diagram in the proof of \cite[Thm. 3]{MV94}.) This shows $\iota(S_{n-1}(u)) = (-q)^{rs(n^2-1)}\gamma_{r,s}(S_{n-1}(y))$.

Next, $K_q(N_K) = \c[y]$ is a right module over $K_q(N_T)$. By the definition of this module structure, the linear map $\mu: K_q(N_T) \to K_q(N_K)$ induced from $\mu: N_T \to N_K$ is given by $\mu(a) = \varnothing \cdot a$, where $\varnothing$ is the empty link. Under the identification $\C[y] = K_q(N_k)$ the empty link corresponds to $1$, which shows that $\mu(a) = 1\cdot a$. This completes the proof.
\end{proof}


\begin{corollary}\label{corollary_topcablingformula}
 The colored Jones polynomials of the cable $K_{r,s}^\tp$ are computed by the formula
 \[
  J_{n}( K_{r,s}^\tp;q) = (-q)^{rs(n^2-1)}\langle \varnothing \cdot \gamma_{r,s}(S_{n-1}(y)), \varnothing \rangle_K
 \]
where the pairing $\langle -,-\rangle_K$ is the pairing associated to the knot $K$ described in (\ref{knotpairing}).
\end{corollary}
\begin{proof}
There are two tautological inclusions $\iota_{r,s}: N_{r,s} \to S^3$ and $\iota_K: N_K \to S^3$. These are related to $\Gamma_{r,s}^\tp$ via the formula $\iota_{r,s} = \iota_K \circ \Gamma_{r,s}^\tp$. These inclusions all induce $\C$-linear maps on skein modules, and we will abuse notation and denote the induced maps with the same notation.

Let $\langle -,-\rangle_{r,s}: K_q(N_{r,s}) \otimes K_q(S^3\setminus K_{r,s}^\tp) \to \c[q^{\pm 1}]$ and $\langle -,-\rangle_K: K_q(N_K) \otimes K_q(S^3\setminus K) \to \c[q^{\pm 1}]$ be the pairings described in (\ref{knotpairing}). By definition, the evaluation $\langle f(u), \varnothing\rangle_{r,s}$ is equal to $\iota_{r,s}(f(u)) \in K_q(S^3) = \c[q^{\pm 1}]$. Similarly, $\langle f(y),\varnothing\rangle_K = \iota_K(f(y))$. 

Now if we combine these statements with Theorem \ref{thm_coloredjonespolys} and Lemma \ref{lemma_topskeincable}, we obtain
\begin{eqnarray*}
 J_n(K_{r,s}^\tp;q) &=& \langle S_{n-1}(u),\varnothing\rangle_{r,s} \\
 &=& \iota_{r,s}(S_{n-1}(u)) \\
 &=& \iota_K(\Gamma_{r,s}^\tp(S_{n-1}(u))) \\
 &=& (-q)^{rs(n^2-1)} \langle  \varnothing \cdot \gamma_{r,s}(S_{n-1}(y)),\varnothing\rangle_K
\end{eqnarray*}
The first equality follows from Theorem \ref{thm_coloredjonespolys} because $K_{r,s}^\tp$ is 0-framed, and the last equality follows from Lemma \ref{lemma_topskeincable} and the definitions of $\iota_K$ and $\langle -,-\rangle_K$. This completes the proof.
\end{proof}

This corollary gives a formula for $J_n(K_{r,s}^\tp)$ in terms of the skein module of $K$. However, the typical cabling formula that appears in the literature gives an expression in terms of Jones polynomials of $K$ - we next derive this from Corollary \ref{corollary_topcablingformula}.
\begin{corollary}
 The colored Jones polynomials of the cable $K_{r,s}^\tp$ are given by the formula
 \[
  J_n(K_{r,s}^\tp;q) = (-q)^{rs(n^2-1)}\sum_{j = (-n+1)/2}^{(n-1)/2}
  q^{-4rj(sj+1)} J_{2sj+1}(K; q)
 \]
\end{corollary}
\begin{proof}
 We need to expand $\langle \varnothing \cdot \gamma_{r,s}(S_{n-1}(y)), \varnothing\rangle_K$ in terms of the Jones polynomials of $K$. To do this, we will use a description of the right $A_q^{\Z_2}$-module $K_q(N_K)$ from \cite[Lemmas 5.4, 5.5]{BS14}. 
 
 Let $M = \C[Y^{\pm 1}]$ be the right $A_q\rtimes \Z_2$ module with action 
 \[
  f(Y)\cdot Y = Yf(Y), \quad f(Y)\cdot X = -f(q^{-2}Y),\quad f(Y)\cdot s = -f(Y^{-1})
 \]
 By \cite[Lemma 5.5]{BS14}, the right $A_q^{\Z_2}$-module $K_q(N_K)$ is isomorphic to $M\e$, where $\e = (1+s)/2$. Because of the sign in the action of $s$, the generator of $M\e$ is $\delta\e := (Y - Y^{-1})\e$, and under this isomorphism, $\delta$ is equal to the empty link $\varnothing \in K_q(N_K)$. In the following computation we use the identity $S_{n-1}(A + A^{-1}) = \sum_{j = (-n+1)/2}^{(n-1)/2} A^{2j}$, which is a standard identity for Chebyshev polynomials.
 \begin{eqnarray*}
  \varnothing \cdot \gamma_{r,s}(S_{n-1}(Y+Y^{-1})) 
  &=& \delta \cdot \e \sum_{j = (-n+1)/2}^{(n-1)/2} \gamma_{r,s}(Y^{2j}) \e\\
  &=& 2Y\e \cdot \sum \left[q^{-rs}X^rY^s\right]^{2j} \e\\
  &=& 2Y\cdot \sum q^{-4rsj^2}X^{2jr}Y^{2sj}\e\\
  &=& 2\cdot \sum q^{-4rsj^2} q^{-4jr} X^{2jr} Y^{2sj + 1}\e\\
  &=& 2\cdot \sum q^{-4jr(sj+1)}Y^{2sj+1}\e \\
  &=& \delta \sum q^{-4jr(sj+1)}S_{2sj}(y)
 \end{eqnarray*}
Now since $K$ is $0$-framed, we can use Theorem \ref{thm_coloredjonespolys} and Corollary \ref{corollary_topcablingformula} to conclude
\[
 J_n(K^\tp_{r,s}) = (-q)^{rs(n^2-1)} \sum_{j = (-n+1)/2}^{(n-1)/2} q^{-4rj(sj+1)} J_{2sj+1}(K;q)
\]
\end{proof}
\begin{remark}
 Since both $K^\tp_{r,s}$ and $K$ are $0$-framed, the equality in the final corollary is exact (up to an overall sign $(-1)^n$), and is not just true up to a power of $q$. We also remark that under the substitutions $q^{-4}_{ours} = q_{theirs}$ and $n_{ours} = 1+b_{theirs}$, this last corollary agrees with \cite[Eq. (4.31)]{CD14}, up to an overall sign. Iterating this formula also gives a formula which is exact up to overall sign because $2sj+1$ has the same parity for any $j$.
\end{remark}

\paragraph{Algebraic cabling}
The \emph{algebraic cabling} procedure is similar to the topological one, but framing is dealt with differently. We will give a more abbreviated discussion that highlights the differences.
\begin{definition}
Let $K \subset S^3$ be a framed knot with \emph{any} framing and let $r,s \in \Z$ be relatively prime. Identify $S^1 \times S^1$ with the boundary of a neighborhood of $K$ such that the first copy of $S^1$ is a meridian of $K$ and the second is the longitude determined by the framing. (We note that since $K$ has \emph{any} framing, its longitude is \emph{not necessarily} the same as the (unique) longitude given by Lemma \ref{lemma_meridianlongitude}). Then the $(r,s)$ \emph{algebraic cable} $K^\alg_{r,s}$ of $K$ is the knot which is the image of the $(r,s)$-curve on $T$ and which is given the framing \emph{parallel to the torus} $T$.   
\end{definition}

Let $\gamma_{r,s} \in \SL_2(\Z)$ satisfy $\gamma(0,1) = (r,s)$. The mapping class group of the torus $T$ is $\SL_2(\Z)$, and this induces an action of $\SL_2(\Z)$ on $K_q(T)$. By construction, we have $\gamma_{r,s}(y) = (r,s)_T$, where $(r,s)_T$ is the $(r,s)$ curve on the torus $T$ and $y$ is the longitude of $K$. We remark that the framing of $(r,s)_T$ is parallel to the torus $T$, so $(r,s)_T$ is isotopic to $K^\alg_{r,s}$ as a framed knot.

\begin{definition}\label{def_algcable}
To give a precise statement, we give names to neighborhoods of $K$, its boundary torus, and its cable.
\begin{enumerate}
\item Let $N_{r,s}$ be a neighborhood of $K^\alg_{r,s}$, and identify $K_q(N_{r,s}) \cong \C[u]$ (as algebras) by setting the generator $u$ to be equal to the framed knot $K^\alg_{r,s}$ (which is \emph{not} $0$-framed). 
\item Let $N_T$ be a neighborhood of $T$. Identify $A_q^{\Z_2}$ with $K_q(N_T)$ by identifying $x = X+X^{-1}$ with the meridian and $y = Y+Y^{-1}$ with the longitude in $T$ which is given by the framing of $K$. (This is \emph{not} the topological longitude given by Lemma \ref{lemma_meridianlongitude}.)
\item Let $N_K$ be a neighborhood of $K$, and identify $K_q(N_K) \cong \C[y]$ (as algebras) by equating $y$ with the framed knot $K$. This means that the framing of the knot $K$ is parallel to the torus $T$. (This is notationally consistent as before.)
\end{enumerate}
\end{definition}
The following tautological inclusions still hold:
\[
N_{r,s} \stackrel{\iota}{\hookrightarrow} N_T \stackrel \mu \hookrightarrow N_K
\] 
We will write 
\begin{equation}\label{def_gammaalg}
\Gamma^\alg_{r,s} := \mu \circ \iota
\end{equation} 
for the composition of these inclusions. By functoriality of skein modules and the identifications above, the map $\Gamma^\alg_{r,s}$ induces a $\c$-linear map $\Gamma^\alg_{r,s}:\C[u] \to \C[y]$.

\begin{lemma}\label{lemma_algskeincable}
The induced $\C$-linear map $\Gamma^\alg_{r,s}: \C[u] \to \C[y]$ is given by 
\[
\Gamma^\alg_{r,s}(S_{n-1}(u)) = 1\cdot \gamma_{r,s}(S_{n-1}(y))
\]
\end{lemma}
\begin{proof}
We first compute the image $\iota(u)$. By definition, $\iota(u)$ is the $r,s$ curve on $T$ which is given the framing parallel to $T$. As remarked above, the element $\gamma_{r,s}(y) \in K_q(T)$ is isotopic to $\iota(u)$. This shows that $\iota(S_{n-1}(u)) = \gamma_{r,s}(S_{n-1}(y))$. The computation of the map $\mu$ is identical to the same computation in Lemma \ref{lemma_topskeincable}.
\end{proof}


\begin{corollary}\label{corollary_algcablingformula}
 The colored Jones polynomials of the cable $K_{r,s}^\alg$ are computed by the formula
 \[
  J_{n}( K_{r,s}^\alg;q) = (-q)^\bullet\langle \varnothing \cdot \gamma_{r,s}(S_{n-1}(y)), \varnothing \rangle_K
 \]
where the pairing $\langle -,-\rangle_K$ is the pairing associated to the knot $K$ described in (\ref{knotpairing}). The exponent in $(-q)^\bullet$ depends on $n$, $r$, $s$ and the framing of $K$.
\end{corollary}
\begin{proof}
The only difference between the proof of this and the proof of Corollary \ref{corollary_topcablingformula} is that the knot $K_{r,s}^\alg$ is not $0$-framed. However, since the Jones-Wenzl idempotent kills cups and caps, the evaluation $\langle S_{n-1}(u),\varnothing\rangle_{r,s}$ will be a power of $-q$ times the Jones polynomial $J_{n}(K_{r,s}^\alg;q)$.
\end{proof}

\subsection{The $\sl_2$ double affine Hecke algebra}\label{sec_dahabackground}
In this section we recall background about the double affine Hecke algebra $\HH_{q,t}$ of type $A_1$ that is required for Cherednik's construction and for our purposes later. The standard reference for the material in this section is \cite{Che05}.

\subsubsection{The Poincar\`e-Birkhoff-Witt property}
We first give a presentation of the algebra $\HH_{q,t}$.
\begin{definition}
Let $\HH_{q, t}$ be the algebra generated by $X^{\pm 1}$, $Y^{\pm 1}$, and $T$ subject to the relations
\begin{equation}
\label{daha}
TXT=X^{-1},\quad TY^{-1}T = Y, \quad XY=q^2YXT^2, \quad
(T-t)(T+t^{-1})=0
\end{equation}
\end{definition}

We remark that we have replaced the $q$ that is standard in the third relation with $q^2$ to agree with the standard conventions for the skein relations in Figure \ref{kbsm}. Also, the fourth relation implies that $T$ is invertible, with inverse $T^{-1} = T + \inv t - t$. Finally, if we set $t=1$, then the fourth relation reduces to $T^2 = 1$, and the third relation becomes $XY=q^2YX$. These imply that $\HH_{q,1}$ is isomorphic to the cross product $A_q\rtimes \Z_2$ (where the generator of $\Z_2$ acts by inverting $X$ and $Y$). 

One of the key propeties of $ \HH_{q,t} $ is the so-called PBW property, which says that,
for all $\,q,t \in \C$, the multiplication map yields a linear isomorphism
$$
\c[X^{\pm 1}] \otimes \c[\Z_2] \otimes \c[Y^{\pm 1}] \stackrel{\sim}{\to} \HH_{q,t}
$$
Another way of stating this property is that the elements $\,\{X^n T^\varepsilon Y^m\, :\, m,n \in \Z\,,\,\varepsilon = 0,\,1\}\,$
form a linear basis in $ \HH_{q,t} $. (See \cite{Che05}, Theorem~2.5.6(a).)

\subsubsection{The spherical subalgebra}\label{sph}
If $ t \ne \pm i $, the algebra $ \HH_{q,t} $ contains the idempotent $\e := (T+t^{-1})/(t+t^{-1})$ (the identity
$\,\e^2 = \e \,$ is equivalent to the last relation in (\ref{daha})). The \emph{spherical subalgebra} of $ \HH_{q,t} $ is
\begin{equation}
\label{salg}
\SH_{q,t} := \e\HH_{q,t}\e 
\end{equation}
Note that $ \A$ inherits its additive and multiplicative structure from $ \HH_{q,t} $, but the identity element of $\A$ is
$ \e $, which is different from $\, 1 \in \HH_{q,t} $. 
If $t^2q^{-2} - t^{-2}q^2$ is invertible, then the next lemma shows 
that $ \A $ is Morita equivalent to $ \HH_{q,t} $; the mutually inverse equivalences are given by
\begin{equation}
\label{mor}
\Mod\,\HH_{q,t} \to \Mod\,\A\, ,\ M \mapsto \e M\ ; \qquad
\Mod\,\A \to \Mod\,\HH_{q,t}\, ,\ M \mapsto \HH_{q,t}\,\e \otimes_{\A} M\ .
\end{equation}

\begin{lemma}
If $t^2q^{-2}-t^{-2}q^2$ is invertible, then $\HH_{q,t}\e \HH_{q,t} = \HH_{q,t}$.
\end{lemma}
\begin{proof}
Define $a := t^{-1}X-tX^{-1}$ and $\pi := YT^{-1}$. The statement then follows from the fact that $X\pi$ is invertible and from the following identities:
\begin{align*}
\pi^2 &= 1\\
\pi X &= q^2 X^{-1} \pi\\
a &=X\e - \e X^{-1} \\
\pi a + t^{-2}q^2 a\pi &= -t^{-1}(t^2q^{-2}-t^{-2}q^2)X\pi
\end{align*}
More precisely, let $I$ be the two-sided ideal generated by $\e$. Then the third identity shows $a$ is in $I$, the final identity follows from the first two, and if $(t^2q^{-2}-t^{-2}q^2)$ is a unit then $I$ contains a unit because of the final identity.
\end{proof}

In the case $t=1$, there is an isomorphism $A_q^{\Z_2} \cong \SH_{q,1}$ given by $w \mapsto \e \bar w \e$, where $w \in A_q^{\Z_2}$ is a symmetric word in $X,Y$, and $\bar w$ is the same word, viewed as an element of $\HH_{q,t}$. For later use we will need a presentation of $\SH_{q,t}$ which we give here. 
\begin{definition}
Let $B'_q$ be the algebra generated by $x,y,z$ modulo the following relations:
\begin{equation}\label{relationsforB'}
[x,y]_q = (q^2-q^{-2})z,\quad
[z,x]_q = (q^2-q^{-2})y,\quad
[y,z]_q = (q^2-q^{-2})x
\end{equation}
Also, define $B_{q,t}$ to be the quotient of $B'_q$ by the additional relation
\begin{equation}\label{casimir_rel}
q^2x^2 + q^{-2}y^2+ q^2z^2 -qxyz= \left( \frac t q - \frac q t\right)^2 + \left( q + \frac 1 q\right)^2
\end{equation}
\end{definition}

\begin{rmk}\label{rmkcascentral}
The element on the left hand side of (\ref{casimir_rel}) is central in $B'_q$ (see Corollary \ref{casimir_central_cor}), so $B_q$ is the quotient of $B'_q$ by a central character.
\end{rmk}

\begin{theorem}[\cite{Ter13}]\label{sphericalrelations}
There is an algebra isomorphism $f:B_{q,t} \to \SH_{q,t}$ defined by the following formulas:
\begin{align}\label{xyz}
x &\mapsto (X+\inv X)\e\notag\\
y &\mapsto (Y+\inv Y)\e\\
z &\mapsto q^{-1}(XYT^{-2}+X^{-1}Y^{-1}) \e\notag
\end{align}
\end{theorem}
\begin{proof}
The fact that (\ref{xyz}) gives a well-defined algebra map can be checked directly, and the fact that it is an isomorphism is proved in \cite{Ter13}. (See also \cite[Thm. 2.20]{BS14} for the precise conversion between Terwilliger's notation and ours.)
\end{proof}

\begin{remark}
 A-priori, it isn't obvious that the elements on the right hand side of (\ref{xyz}) are contained in $\e \HH_{q,t} \e$. However, short computations show that if we take $a \in \HH_{q,t}$ to be either $X+X^{-1}$, $Y+Y^{-1}$, or $XYT^{-2}+X^{-1}Y^{-1}$, then $a\e = \e a$, and this implies $a \e = \e a \e \in \e \HH_{q,t}\e$.
\end{remark}

For later reference, we include a lemma that is useful for establishing isomorphisms of $B'_q$-modules.
\begin{lemma}\label{lemma_invsubalgebraiso}
 Suppose that $M$ and $N$ are modules over $B'_q$, and that as a $\C[x]$-module $M$ is generated by elements $\{m_i \in M\}$. Furthermore, suppose that $f:M \to N$ is a morphism of $\C[x]$-modules that satisfies $f(ym_i) = yf(m_i)$ and $f(zm_i) = zf(m_i)$. Then $f$ is a morphism of $B'_q$-modules.
\end{lemma}
\begin{proof}
 By definition, the elements $x,y,z \in B'_q$ satisfy the commutation relations (\ref{relationsforB'}).
An arbitrary element of $M$ can be written as $m = \sum_{i=1}^n c_i p_i(x) m_i$, and using the $\C[x]$-linearity of $f$ and the commutation relations, powers of $x$ in the expressions $ym$ and $zm$ can inductively be moved to the left. This shows that $f(ym) = yf(m)$ and $f(zm) = zf(m)$ for arbitrary $m \in M$, which completes the proof.
\end{proof}

%
%

%
\subsubsection{The standard and sign polynomial representation}
\label{dunkl}
Our definition of $\HH_{q,t}$ was in terms of generators and relations. However, $\HH_{q,t}$ can also be viewed as a family of subalgebras of $\End_\C(\C[X,X^{-1}])$ using the construction 
we describe in this section. We first introduce the following linear operators on $\c[X^{\pm 1}]$:
\begin{equation}
\label{op}
\sx[f(X)] := X f(X)\ ,\quad \s[f(X)] := f(X^{-1})\ ,\quad \sy[f(X)] := f(q^{-2}X)\  ,
\end{equation}
Notice that these operators are invertible and satisfy the relations
\begin{equation}
\label{rop}
\s^2 =1 \ , \quad
\s\,\sx = \sx^{-1} \s \ , \quad
\s\,\sy = \sy^{-1}\s \ ,\quad
\sx\,\sy = q^2\,\sy\,\sx\ .
\end{equation}
Thus, they define a representation of the crossed product $A_q \rtimes \Z_2$. The action of $ \HH_{q,t} $ on $ \c[X^{\pm 1}]$
can be described by
\begin{equation}
\label{du}
X \mapsto \sx \ ,\quad
T \mapsto \sT := t\cdot\s + \frac{t-t^{-1}}{X^2-1}(\s-1) \ ,\quad
Y \mapsto  \sy\,\s\,\sT
\end{equation}
The operator $ \sT $ is called the {\it Demazure-Lusztig operator}
(cf. \cite{Che05}, (1.4.26), (1.4.27)). Formally it is an operator on rational functions $\c(X)$, but  it preserves the subspace $\c[X^{\pm 1}] \subset \C(X)$ because $f(X^{-1})-f(X)$ is always divisible (in $\c[X^{\pm 1}]$) by $X^2-1$.

These assignments can be rephrased as follows: let
$\D_q$ be the localization of $A_q\rtimes \Z_2$ with respect to the multiplicative set consisting of all nonzero polynomials in $X$. Then formulas \eqref{du} give an embedding
\begin{equation}
\label{duem}
\Theta_{q,t}:\, \HH_{q,t} \into  \D_q
\end{equation}

The \emph{sign representation} is defined similarly - we first define operators
\begin{equation}
\sx[f(X)] := X f(X)\ ,\quad \s_-[f(X)] := -f(X^{-1})\ ,\quad \sy_-[f(X)] := -f(q^{-2}X)
\end{equation}
These operators give $\C[X^{\pm 1}]$ the structure of an $A_q\rtimes\Z_2$-module. We then define
\begin{equation}\label{eq_signde}
X \mapsto \sx \ ,\quad
T \mapsto \sT_- := -t^{-1}\cdot\s_- + \frac{t-t^{-1}}{X^2-1}(\s_-+1) \ ,\quad
Y \mapsto  \sy_-\,\s_-\,\sT_-
\end{equation}
It can be checked directly that these assignments give an embedding $\Theta^-_{q,t}:\H_{q,t} \to \D_q$.

\begin{definition}
 Let $P^+,P^- := \C[X^{\pm 1}]$ be the $\H_{q,t}$-modules given by formulas (\ref{du}) and (\ref{eq_signde}), respectively. We refer to these as the \emph{polynomial representation} and the \emph{sign representation}.
\end{definition}

\begin{remark}
 The modules $P^+$ and $P^-$ can also be viewed as induced modules as follows. Let $\H_Y \subset \H_{q,t}$ be the subalgebra generated by $Y^{\pm 1}$ and $T$. Then $\H_Y$ is the affine Hecke algebra (of type $A_1$) and it has two natural 1-dimensional modules, $\C_t$ and $\C_{-t^{-1}}$. On the module $\C_t$ both $Y$ and $T$ act by multiplication by $t$, and on $\C_{-t^{-1}}$ they act by multiplication by $-t^{-1}$. Then $P^+$ and $P^-$ are the $\H_{q,t}$-modules which are induced from $\C_t$ and $\C_{-t^{-1}}$, respectively.
\end{remark}

\subsubsection{The symmetric polynomial representation}\label{sympolyrep}
Under the Morita equivalence \eqref{mor}, the $ \HH_{q,t}$-module $ P^+ $ corresponds to the $ \A$-module
$$
\e P^+ = \e\H_{q,t}/(\e\HH_{q,t} \, (T-t) + \e\HH_{q,t}\, (Y-t)) \cong \e\cdot \c[X^{\pm 1}]
$$
Now, note that the subspace $ \e\c[X^{\pm 1}] \subset \c[X^{\pm 1}] $ is the image of the projector $ \e = (t^{-1}+T)/(t^{-1}+t)$
and hence the kernel of $1 - \e$. By the Bernstein-Zelevinsky lemma (see \cite{Che05}, p.~202),
the kernel of the operator $ T - t $ (acting on $ \c[X^{\pm 1}] $ as in \eqref{du}) is exactly $\c[X^{\pm 1}]^{\Z_2} = \c[X+X^{-1}]$, the subspace of symmetric Laurent polynomials. 
Thus, for {\it all} parameters $q,t \in \C^*$, the spherical algebra  $ \A $ acts on $ \c[X+X^{-1}] $ via the
identification
\begin{equation}
\label{sphr}
\c[X+X^{-1}] = \e\cdot \c[X^{\pm 1}] \cong \e P^+ \quad
f(X+X^{-1}) \leftrightarrow \e f(X+X^{-1}) \leftrightarrow [\e f(X+X^{-1})]
\end{equation}

Since $Y+Y^{-1}$ commutes with $\e$ it preserves the subspace $\C[X+X^{-1}]\subset \C[X^{\pm 1}]$. This operator is called the {\it Macdonald operator}, and it plays a fundamental role in the representation-theoretic approach
to the theory of Macdonald polynomials. A computation shows that this operator can be written as
\begin{equation}
\label{mac}
L_{q,t} := Y + Y^{-1} = \frac{tX^{-1}-t^{-1}X}{X^{-1}-X}\,\sy + \frac{t^{-1}X^{-1}-tX}{X^{-1}-X}\,\sy^{-1}\ .
\end{equation}
\begin{remark}
 Under the identification $\SH_{q,t=1} \cong K_q(T^2)$ in Theorem \ref{fg00}, the Macdonald operator is identified with the longitude of the torus.
\end{remark}

\subsubsection{Rank 1 Macdonald polynomials}
\label{macd}
We briefly review the definition of Macdonald polynomials of type $A_1$ (a.k.a.~ the Rogers or
$(q,t)$-ultraspherical polynomials, see \cite{AI83}). This family of orthogonal polynomials depends on two parameters $ q,t \in \C^*$
and forms a basis in the space $ \c[X+X^{-1}] = \C[x] $ of symmetric Laurent polynomials\footnote{The Macdonald polynomials of type $A_1$ constitute a subfamily of the four-parameter family of the 
so-called {\it Askey-Wilson polynomials}, which is the most general family of orthogonal
polynomials of one variable. The Askey-Wilson polynomials are controlled by the DAHA of type $CC^\vee$.}. These polynomials can be naturally defined for an arbitrary root system; 
they possess some remarkable
properties which were first conjectured by I.~G.~Macdonald and proved by Cherednik (see \cite[Sect. 1.4]{Che05} for a discussion of these polynomials and the Macdonald conjectures).

The Macdonald polynomials can be defined as solutions of the eigenvalue problem for the Macdonald operator:
$$
L_{q,t}[\varphi] = (\lambda + \lambda^{-1})\,\varphi
$$
\begin{theorem}
\label{thdef}
There is a unique family of symmetric Laurent polynomials $\{p_0(x),\,p_1(x),\,\ldots\} \subset \c[X^{\pm 1}]^{\Z_2}=\C[x] $, satisfying
\begin{eqnarray*}
&(a)& L_{q,t}[p_n] = (tq^{2n} + t^{-1} q^{-2n})\,p_n \ ,\\
&(b)& p_0(x) = 1 \ ,\quad p_n(x) = X^n + X^{-n} + \sum_{|m| < n}\,c_m\,x^m\ ,\ n \ge 1
\end{eqnarray*}
\end{theorem}
With this choice of normalization, the Macdonald polynomials depend rationally on $q,t$. For example, the first members of this family are
$$
p_0(x) = 1\ ,\quad p_1(x) = X + X^{-1} \ ,\quad p_2(x) = X^2 + X^{-2} + (1-t^2)(1+q^4)/(1- t^2 q^4)
$$

\subsubsection{The Dunkl-Cherednik pairing}\label{sec_pairings}
Here we recall an important property of the polynomial representation which is analogous to the Shapovalov form from Lie theory. First we define an anti-automorphism of $\H_{q,t}$:
\begin{equation}\label{eq_antiautomorphism}
\phi: \HH_{q,t} \to \HH_{q,t}\ ,\quad X \mapsto Y^{-1}\ ,\quad Y \mapsto X^{-1}\ , \quad T \mapsto T\ ,
\end{equation}
We write $\phi(P^+)$ for the twist of the module $P^+$ by this anti-automorphism. 

\begin{lemma}\label{lemma_prepairing}
 The vector space $\phi(P^+) \otimes_{\H_{q,t}} P^+$ is 1-dimensional.
\end{lemma}
\begin{proof}
 This follows from the PBW property combined with some computations (see \cite[Sec. 1.4.2]{Che05}).
\end{proof}

\begin{corollary}\label{cor_pairings}
 There are pairings
 \[
  \langle -,-\rangle_{q,t}: \phi(P^+) \otimes_{\H_{q,t}} P^+ \to \C,\quad,\quad
  \langle -,-\rangle_{q,t}: \phi(P^+\e) \otimes_{\SH_{q,t}} \e P^+ \to \C
 \]
 and these pairings are both uniquely determined by the condition $\langle 1,1\rangle = 1$.
\end{corollary}
\begin{proof}
 The existence and uniqueness of the first pairing follows from Lemma \ref{lemma_prepairing}. Since $\SH_{q,t}$ and $\H_{q,t}$ are Morita equivalent, the natural map $\phi(P^+\e) \otimes_{\SH_{q,t}} \e P^+ \to \phi(P^+) \otimes_{\H_{q,t}} P^+$ is an isomorphism of vector spaces (see, e.g. \cite[Lemma 5.2]{BS14}). This shows the existence and uniqueness of the second pairing.
\end{proof}

\begin{remark}
 This pairing is very closely related to the topological pairing (\ref{knotpairing}). More precisely, if the construction in this section is repeated for the sign representation $P^-$, then this pairing at $t=1$ is \emph{exactly} the same as the pairing (\ref{knotpairing}) when $L$ is the unknot (see Lemma \ref{lemma_teq1ispairing}).
\end{remark}

\subsubsection{The sign representation}\label{sec_signrep}
In this subsection we recall some facts about the sign representation $P^-$ of $\HH_{q,t}$, which will be used in Section \ref{sec_topologicalunknot}. 

The symmetric sign representation $\e P^-$ can be identified with $\C[x]$ as follows. First, we define 
\[ \delta_t := tX^{-1}-t^{-1}X\] 
and note that in the sign representation we have $\e \delta_t = \delta_t$ and $\e\cdot 1 = 0$. Since 
$T$ commutes with elements of $\C[X+X^{-1}] \subset \HH_{q,t}$, the $-t^{-1}$ and $t$ eigenspaces of $T$ are
\[
 P^- = \C[X+X^{-1}] \oplus \C[X+X^{-1}]\delta_t
\]
Therefore, $\e$ kills the first factor in this decomposition, and we obtain
\begin{equation}\label{eq_symsigndecomp}
 \e P^- = \C[x]\delta_t \subset \C[X^{\pm 1}]
\end{equation}
where $x = X + X^{-1}$. As in the previous section, we write $P^-\e$ for the right $\HH_{q,t}$-module that is the twist of $\e P^-$ by the anti-automorphism $\phi$ from (\ref{eq_antiautomorphism}).

\begin{lemma}\label{signpairing}
 There is a well-defined pairing $\langle -,-\rangle: P^-\e \otimes_{\SH_{q,t}} \e P^- \to \C$, and this pairing is uniquely determined by the condition $\langle \delta_t, \delta_t\rangle = 1$.
\end{lemma}
\begin{proof}
 The uniqueness of the claimed pairing follows from the decomposition (\ref{eq_symsigndecomp}) and from the fact that $\delta_t$ is an eigenvector of $\e(Y+Y^{-1})\e$. To show existence, we first note that there is an isomorphism $f: \HH_{q,t} \to \HH_{q,-t^{-1}}$ determined by 
 \[
  f(X) = X,\quad f(Y) = Y,\quad f(T) = T
 \]
 The twist $f^*(P^+)$ of the standard polynomial representation by this isomorphism is the sign representation, and combining this with Lemma \ref{lemma_prepairing} shows that $\phi(P^-)\otimes_{\H_{q,t}} P^-$ is one-dimensional. Then the claim follows using the same argument as in Corollary \ref{cor_pairings}.
\end{proof}

We will also need the analogues of the Macdonald polynomials for the sign representation.
\begin{definition}\label{def_signmacpolys}
 We define the \emph{sign Macdonald polynomials} by $p_n^-(x) = p_n^-(x;q,t) := p_n(x;q,-t^{-1})$.
\end{definition}
\begin{remark}\label{rmk_signmacpolys}
By the proof of Lemma \ref{signpairing}, the polynomials $p_n^-(x)\delta_t \in \e P^-$ are eigenvectors for the Macdonald operator $Y+Y^{-1}$.
\end{remark}

\section{A modified Kauffman bracket skein module} \label{sec_topologicaldeformations}

In this section we discuss deformations of skein modules to modules over the DAHA $\H_{q,t}$ of type $A_1$. We use a modification of the Kauffman bracket skein module to give a topological construction of the spherical subalgebra $\SH_{q,t}$. We also give a construction which associates an $\SH_{q,t} $-module to each knot in $S^3$. 

\subsection{Modified KBSM for surfaces}\label{mkbsm}
Our goal in this section is to define a 2-parameter skein algebra $K_{q,t}(F)$ for a (connected) surface $F$. As we describe below, the algebra $K_{q,t}(F)$ will be a quotient of the skein module $K_q(F \setminus \{p\})$ of the punctured surface by an ideal depending on the parameter $t \in \C^*$. For general surfaces there is a surjection $K_{q,t=1}(F) \twoheadrightarrow K_q(F)$, but the specialization $K_{q,t=1}(F)$ is `bigger' than $K_q(F)$. More precisely, the set of links with non-crossing, nontrivial components forms a (linear) basis of $K_q(F)$ but does not span $K_{q,t=1}(F)$. However, if $F$ is the torus $T^2$, then this set \emph{is} a linear basis for $K_{q,t}(T^2)$ (for all $t$). Therefore, $K_{q,t}(T^2)$ can be viewed as a (flat) deformation of $K_q(T^2)$. We will show that this deformation is the same as the algebraic deformation given by the spherical subalgebra $\SH_{q,t}$ described previously.

\begin{definition}
We fix a point $p \in F$ and define the following: 
\begin{enumerate}
\item a \textbf{vertical special strand} is a homeomorphism $[0,1] \to \{p\} \times [0,1] \subset F \times [0,1]$,
\item a \textbf{framed link with a vertical special strand} is an isotopy class of embeddings $[0,1] \sqcup \left(\sqcup_n S^1\times[0,1]\right)\hookrightarrow F \times [0,1]$, where $n \geq 0$ and the isotopy class contains a representative whose embedding of $[0,1]$ is a vertical special strand. (The allowed isotopies must fix the boundary of the 3-manifold, and in particular the endpoints of the special strand are fixed.)
\end{enumerate}
\end{definition}
Let $\mathcal L_{q,t}$ be the vector space spanned by framed links with a vertical special strand. (We remark that $\mathcal L_{q,t}$ is isomorphic to the space spanned by framed links in a thickening of the punctured surface. However, we phrase the definition using `special strands' because we will later want to extend this construction to 3-manifolds that are not thickened surfaces, and it seems less convenient to use the `puncture' definition in this situation.)

Let $\mathcal L'_{q,t}\subset \C \mathcal L_{q,t}$ be the subspace generated by elements of the form
\begin{equation}\label{modifiedkbsm_rel}
L_+-qL_0-q^{-1}L_\infty, \quad (L \sqcup \bigcirc) + (q^2  + q^{-2})L, \textrm{  and  } L_s+(q^2t^{-2}+q^{-2}t^2)L'
\end{equation}
where $L_+,L_0,L_\infty$ are links that are identical outside of a ball, and inside the ball appear as the first, second, and third terms of Figure \ref{kbsm} (respectively). Similarly, $L_s,L'$ are links which are identical outside a ball, and inside a ball appear as in Figure \ref{kbsm_newrel}. (The dotted lines in Figure \ref{kbsm_newrel} represent the special strand, and the solid loop in $L_s$ is a component of a framed link in the skein module.) We note that the third relation in \eqref{modifiedkbsm_rel} doesn't require the special strand to be framed, since the relation doesn't use the framing of the special strand in any way.

\begin{definition}
The \emph{modified Kauffman bracket skein module} $K_{q,t}(F)$ is the algebra $\C \mathcal L_{q,t} / \mathcal L'_{q,t}$.
\end{definition}

\begin{remark}\label{rmk:fillin} In words, the first two relations in \eqref{modifiedkbsm_rel} say that the standard skein relations apply between links in $F \times [0,1]$.  The third relation (between $L_s$ and $L$) is saying ``a loop around the special strand can be removed at the cost of a constant.'' In other words, $K_{q,t}(F)$ is isomorphic to $K_{q}(F \setminus p)$ modulo the relation that says: if $L_p$ is a horizontally-framed loop encircling the puncture $p$, then $L_p = -q^2t^{-2} - q^{-2}t^2$.
\end{remark} 


\begin{figure}
\begin{center}
\input{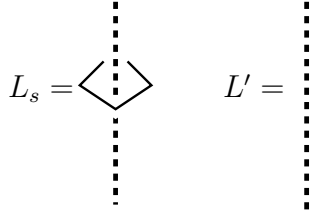}
\caption{Skein relation for special strands}\label{kbsm_newrel}
\end{center}
\end{figure}

The algebra structure is given by ``stacking links" as before. More precisely,  given two links $L_1,L_2 \subset F \times [0,1]$, we take two copies of $F \times [0,1]$, each containing one of the $L_i$, and glue $F \times \{1\}$ in one to $F\times \{0\}$ in the other via the identity map. Since each special strand begins and ends at $p$, the special strands glue together. The identity element of this algebra is the link with one special strand and no other components. We also remark that the identity component of the diffeomorphism group of a surface acts transitively, so different choices of $p$ give isomorphic algebras. We therefore do not include the point $p \in F$ in the notation.

\begin{lemma}
There is a natural surjective algebra map $K_{q,t=1}(F) \twoheadrightarrow K_{q}(F)$.
\end{lemma}
\begin{proof}
First, there is a natural map $f: K_q(F \setminus p) \to K_q(F)$. Let $L_p \in K_q(F\setminus p)$ be the loop which encircles the puncture $p$. Under this map, $L_p$ is sent to a nullisotopic loop, which is equal to $-q^2-q^{-2}$ by the (standard) skein relations. Second, if $t=1$, then $L_p = -q^2t^{-2}-q^{-2}t^2 = -q^2-q^{-2}$ in $K_{q,t=1}(F)$. Third, by Remark \ref{rmk:fillin}, $K_{q,t}(F)$ is the  quotient of $K_q(F\setminus p)$ by the relation $L_p = -q^2t^{-2}-q^{-2}t^2$. Combining these three statements shows that if $t=1$, then the map $f: K_q(F \setminus p) \to K_q(F)$ factors through to the quotient, so induces a map $\tilde f: K_{q,t=1}(F) \to K_q(F)$. The map $\tilde f$ is surjective because $f$ is.
\end{proof}

\subsection{The torus}
We now compute the algebra structure of $K_{q,t}(T^2)$. We first recall a useful result from \cite{BP00}. Let $T'$ be the torus with one puncture, and let $x'$ and $y'$ be simple closed curves that intersect once, and let $z'$ be the (simple closed) curve with the coefficient $q$ in the resolution of the product $x'y'$. (See Figure \ref{genpunctorus}.)

\begin{figure}
\begin{center}
\input{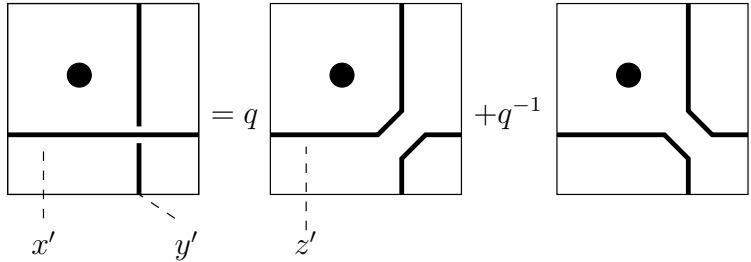}
\caption{The generators for the punctured torus}\label{genpunctorus}
\end{center}
\end{figure}

We also need the algebras $B'_q$ and $B_{q,t}$ from Section \ref{sph}. We recall
that $B'_q$ is the algebra generated by $x,y,z$ modulo the following relations:
\begin{equation}\label{relationsforB'_top}
[x,y]_q = (q^2-q^{-2})z,\quad
[z,x]_q = (q^2-q^{-2})y,\quad
[y,z]_q = (q^2-q^{-2})x
\end{equation}
and that $B_{q,t}$ is the quotient of $B'_q$ by the additional relation
\begin{equation}\label{casimir_rel_top}
q^2x^2 + q^{-2}y^2+ q^2z^2 -qxyz= \left( \frac t q - \frac q t\right)^2 + \left( q + \frac 1 q\right)^2
\end{equation}

\begin{theorem}\textup{\cite[Thm 2.1]{BP00}}\label{thm_bp}
There is an algebra isomorphism $B'_q \to K_q(T')$ induced by the assignments $x \mapsto x'$, $y \mapsto y'$, and $z \mapsto z'$.
\end{theorem}

\begin{remark}\label{boundaryparallelcentral}
It is a general fact that if $F$ is any surface and $\partial$ is a curve parallel to a boundary component of $F$, then $\partial$ is a central element in $K_q(F)$. This is true because if $\partial$ is `on top of' a link $L$, then $\partial$ can be shrunk to be very close to the boundary, slid down the boundary until it is below $L$, and then expanded. In other words, the link $\partial L$ is isotopic to the link $L\partial$.
\end{remark}

\begin{corollary}\label{casimir_central_cor}
The element $w = q^2x^2 + q^{-2}y^2+ q^2z^2-qxyz$ is central in $B'_q$.
\end{corollary}
\begin{proof}
If we show that the element $-w+q^2+q^{-2}$ is the loop parallel to the boundary of $T'$, then the considerations in the previous remark apply to prove the claim. To see this identity we use \cite[Eq. (2)]{BP00}. Write $\delta$ for the loop around the puncture and $\alpha$ for the curve with coefficient $q^{-1}$ in Figure \ref{genpunctorus}. Then in our notation their identity is
\[
\alpha z = q^2 x^2 + q^{-2} y^2 - q^2 - q^{-2} + \partial
\]
Then the identity $xy = qz + q^{-1} \alpha$ of Figure \ref{genpunctorus} shows that 
\[
qxyz = q^2z^2 + (q^2 x^2 + q^{-2} y^2 - q^2 - q^{-2} + \partial)
\]
This shows that $-w+q^2+q^{-2} = \partial$, as desired.
\end{proof}

It is clear that there is a surjection $K_q(T') \twoheadrightarrow K_{q,t}(T^2)$ induced topologically by `filling in the puncture with the special strand.' This means that we can consider the loops $x',y',z'$ as elements of $K_{q,t}(T^2)$.

\begin{corollary}\label{corbqdkbsm}
There is an algebra isomorphism $B_{q,t} \to K_{q,t}(T^2)$ induced by the assignments $x \mapsto x'$, $y \mapsto y'$, and $z \mapsto z'$.
\end{corollary}
\begin{proof}
First we check that the composition $f:B'_q \to K_q(T') \to K_{q,t}(T^2)$ factors through the quotient $B'_q \to B_{q,t}$. To do this we need to check that relation (\ref{casimir_rel}) holds, i.e. that $f(w) = \left( \frac t q - \frac q t\right)^2 + \left( q + \frac 1 q\right)^2$. If $\partial \in K_q(T')$ is the loop around the puncture in $T'$, then $\partial = -w + q^2 + q^{-2} \in K_q(T')$ (see the proof of Corollary \ref{casimir_central_cor}). Then the third relation in (\ref{modifiedkbsm_rel}) implies $f(\partial) = -q^2t^{-2}-q^{-2}t^2$, and the calculation $\left(t/q-q/t\right)^2+(q+q^{-1})^2 = q^2+q^{-2}+q^2t^{-2}+q^{-2}t^2$ shows the claim.

To show that $f:B_{q,t} \to K_{q,t}(T^2)$ is injective, it suffices to show that the kernel of the map $K_q(T') \to K_{q,t}(T^2)$ is cyclic and is generated by $\delta + q^2t^{-2}+q^{-2}t^2$, where $\delta$ is the loop in $T'$ that encircles the puncture. This element is exactly the skein relation on the left of Figure \ref{kbsm_newrel}, and any time this relation appears in a sum of links, we can slide the relation to the top of the manifold $T'\times [0,1]$. In other words, any element $a$ in the kernel of $K_q(T') \to K_{q,t}(T^2)$ can be written in the form $a = a'(\delta + q^2t^{-2}+ q^{-2}t^2)$ for some $a' \in K_q(T')$. This proves $\delta + q^2t^{-2}+ q^{-2}t^2$ generates the kernel, as desired. 
\end{proof}

\begin{corollary}\label{cor_surfacedef}
The algebras $\SH_{q,t}$ and $K_{q,t}(T^2)$ are isomorphic.
\end{corollary}
\begin{proof}
Compose the isomorphism of Corollary \ref{corbqdkbsm} and Theorem \ref{sphericalrelations}.
\end{proof}

\subsection{The deformation for 3-manifolds}\label{sec_3mandefs}
We now define a deformed skein module $\bar K_{q,t}(M,f)$ for a 3-manifold $M$ with a boundary component $F$, a connected surface. The vector space $\bar K_{q,t}(M, f)$ is actually a bimodule - it is a left module over $K_{q,t}(F)$ and a right module over $K_{q,t}(T^2)$. The definition of $\bar K_{q,t}(M, f)$ depends on some additional data (the map $f$, which is described below), but if $M$ is a knot complement $S^3 \setminus K$, then we describe a canonical choice for this data.

For knot complements we also define a canonical quotient $K_{q,t}(S^3\setminus K)$ of $\bar K_{q,t}(S^3\setminus K)$. This quotient destroys the right module structure, so $K_{q,t}(S^3\setminus K)$ is just a left module over $K_{q,t}(T^2)$. When $K$ is the unknot this module has the `right size' and can be viewed as a deformation of $K_q(S^3 \setminus K)$. However, it is not clear if this is true for other knots.

\subsubsection{General 3-manifolds}
Let $G$ be the `tadpole' graph depicted in Figure \ref{fig:tadpole}, and let $f:G \to M$ be an embedding with $f(v_1) = p \in F \subset \partial M$. (Here $v_1$ is the vertex of $G$ that is not on the loop.) As before, the definition doesn't depend on the choice of $p \in F$, but it \emph{does} depend on the choice of $f$.
\begin{definition}\label{def:3manqt}
A \textbf{framed link compatible with $f$} is an isotopy class of embeddings $G \sqcup (\sqcup_n S^1\times[0,1])\hookrightarrow M$ containing a representative such that the embedding of $G$ is given by the map $f$. (The isotopies we consider are those which fix the boundary of $M$.) The \textbf{modified Kauffman bracket skein module} $\bar K_{q,t}(M,f)$ is the vector space of framed links compatible with $f$ modulo the relations in Equation (\ref{modifiedkbsm_rel}), where the third relation in \eqref{modifiedkbsm_rel} is only applied to the arc of the graph $G$, and not the loop of $G$. 
\end{definition}

\begin{remark}\label{rmk:3manrephrasing}
Definition \ref{def:3manqt} could be rephrased as follows. Let $M' := M \setminus (\textrm{nbhd of } f(G))$. Then $\bar K_{q,t}(M)$ is the quotient of the skein module $K_q(M')$ by the relation that says: if $L$ is a loop in $K_q(M')$ that is parallel to the loop on the boundary of $M'$ that encircles the arc in $G$, then $L = -q^2t^2 - q^{-2}t^{-2}$. This rephrasing shows that the definition of $\bar K_{q,t}(M,f)$ is invariant under isotopies of the map $f$.
\end{remark}

\begin{figure}
\begin{center}
\input{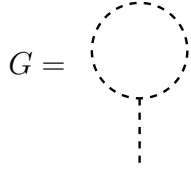}
\caption{The `tadpole' graph}\label{fig:tadpole}
\end{center}
\end{figure}

\begin{lemma}
The space $\bar K_{q,t}(M, f)$ is a left module over $K_{q,t}(\partial M)$.
\end{lemma}
\begin{proof}
To see this, we note that there is a natural inclusion of the punctured surface $\partial M \setminus p$ into $M$, which means that $\bar K_{q,t}(M,f)$ is a (left) module over $K_q( (\partial M) \setminus p)$. By Remark \ref{rmk:fillin}, $K_{q,t}(\partial M)$ is a quotient by $K_q((\partial M)\setminus p)$ by the relation $L_p = -q^2t^{-2}-q^{-2}t^2$. By definition, this relation holds in $\bar K_{q,t}(M,f)$, so the action of $K_q((\partial M)\setminus p)$ factors through to an action of $K_{q,t}(\partial M)$ on $\bar K_{q,t}(M,f)$.
\end{proof}

If we thicken the loop in the graph $G$, then the boundary $B$ of this thickened loop can be identified with the punctured torus $T^2 \setminus p$. Therefore, $\bar K_{q,t}(M, f)$ is also a right module over $K_{q,t}(T^2)$. However, in general this right module structure is not canonical - it depends on the identification of $T^2 \setminus p$ with the boundary $B$ (and also on the map $f$).


\subsubsection{Knot Complements}\label{sec_knotcomplementsmodule}
Let $K \subset S^3$ be a knot. In this section we describe a canonical choice of a map $f:G \to S^3 \setminus K$ (see Figure \ref{tadpole}). We also describe a quotient $K_{q,t}(K)$ of $\bar K_{q,t}(S^3\setminus K, f)$ that is closely related to the classical skein module $K_q(S^3\setminus K)$. In the next section we compute this module when $K$ is the unknot.

If $K$ is a knot in $S^3$, then there is a unique (up to isotopy) choice of a longitude and meridian in the boundary of a tubular neighborhood of $K$. There is therefore a unique (up to isotopy) choice of curve $C$ in $S^3\setminus K$ that is parallel to the meridian $m$. More precisely, $C$ is a boundary component of an annulus whose other boundary component is the meridian of $K$. We pick a simple arc $a$ in this annulus connecting $C$ to the point $p \in \partial (S^3\setminus K)$, and we define $f:G \to S^3\setminus K$ to be the map given by the union of the embeddings of $C$ and $a$. (See Figure \ref{tadpole}.) 

\begin{figure}
\begin{center}
\input{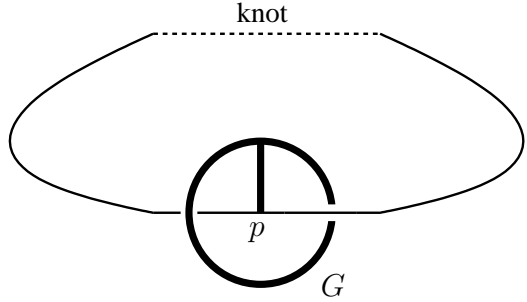}
\caption{The embedding of the graph $G$ for a knot complement}\label{tadpole}
\end{center}
\end{figure}

\begin{definition}
 If $K \subset S^3$ is a knot, then $\bar K_{q,t}(K) := \bar K_{q,t}(S^3\setminus K, f)$, where the map $f:G \to S^3\setminus K$ is the map described technically in the previous paragraph (and graphically in Figure \ref{tadpole}).
\end{definition}

\begin{remark} \label{rmk:arcchoices}
There are many choices of an arc in an annulus connecting the two boundary components, but the definition does not depend on this choice because all such choices are isotopic (via an isotopy that does not fix the embedding of $G$). Such isotopies are explicitly allowed in the definition of a `framed link compatible with $f$' (see the first sentence of Definition \ref{def:3manqt} and the last sentence of Remark \ref{rmk:3manrephrasing}).
\end{remark}

As was mentioned in the previous section, the vector space $\bar K_{q,t}(K)$ is a bimodule over $K_{q,t}(S^1 \times S^1)$. In general, the right module structure is non-canonical, but with our specific choice of $f:G \to S^3 \setminus K$, there is a canonical choice for the right module structure. 

In more detail, we let $N_G$ be a closed tubular neighborhood of the loop in the graph $G$, and we let $T_p$ be the boundary of $N_G$ (which is a punctured torus, where the puncture corresponds to the arc in the graph $G$). Up to isotopy there is a unique loop $m_G$ in $T_p$ that is contractible in $N_G$. By construction, there is a loop $l_G$ in $T$ that is isotopic to the meridian of $K$, and up to isotopy this loop is unique. (The loops $m_G$ and $l_G$ are the meridian and longitude of $G$ if $G$ is viewed as the unknot inside of $S^3$.) The choice of the (oriented) loops $m_G$ and $l_G$ gives an identification of $S^1 \times S^1$ with $T$, and up to isotopy this identification is uniquely determined by the requirement that the first and second factors of $S^1\times S^1$ are sent to $m_G$ and $l_G$, respectively.

We recall that for generic $q$, the algebra $\SH_{q,t}$ is generated by elements $x,y \in \SH_{q,t}$, with $x = (X+X^{-1})\e$ and $y = (Y+Y^{-1})\e$. Under the identifications in the previous paragraph, the elements $x$ and $y$ act on $\bar K_{q,t}(K)$ on the left via $m$ and $l$, respectively. The action of $x$ and $y$ on the right is given by $l_G$ and $m_G$, respectively. (Note that the meridian $m$ of $K$ is isotopic to the longitude $l_G$ of $G$, so the element $x \in \SH_{q,t}$ acts on the empty link $1_L \in \bar K_{q,t}(K)$ symmetrically on the left and on the right.) Let $z_G \in K_{q,t}(T^2)$ be the loop in the boundary of $G$ that has the coefficient $q$ in the resolution of the product $l_Gm_G$, so that $z \in \SH_{q,t}$ acts on $\bar K_{q,t}(K)$ on the right by $z_G$.

\begin{definition}\label{def_leftmodquotient}
 Let $\e P^-$ be the symmetric sign representation, and define the left $\SH_{q,t}$-module
 \[
  K_{q,t}(K) := \bar K_{q,t}(K) \otimes_{\SH_{q,t}} \e P^-
 \]
\end{definition}
The module $K_{q,t}(K)$ can equivalently be defined as the quotient of $\bar K_{q,t}(K)$ by the left submodule generated by $L(m_G + tq^{-2}+t^{-1}q^2)$ and $L(z_G + tq^{-3}l_G)$, for all $L \in \bar K_{q,t}(K)$. Topologically, this definition can be viewed as the $t$-analogue of `filling in the graph $G$'. More precisely, we have the following lemma.

\begin{proposition}\label{prop_teq1surj}
 If $t=1$, there is a natural surjection $K_{q,t=1}(K) \twoheadrightarrow K_q(K)$.
\end{proposition}
\begin{proof}
 There are natural surjections from $K_q(S^3 \setminus (K\cup G))$ onto $K_{q,t}(K)$ and $K_q(K)$. When $t=1$, all relations imposed in the definition of $K_{q,t=1}(K)$ are also in the kernel of the map $K_q(S^3\setminus (K \cup G)) \to K_q(K)$. (This uses the fact that the specialization $\e P^-_{q,t=1}$ is isomorphic to the skein module of the solid torus.) Therefore this map induces a map $K_{q,t=1}(K) \to K_q(K)$.
\end{proof}

\begin{remark}
  The relations defining $K_{q,t}(K)$ can be viewed as skein relations in some sense, but an important difference is that they are not local. For example, any loop isotopic to the $(1,1)$ curve in the torus $T_p$ that bounds the loop in $G$ is declared to be equal to a loop isotopic to $l_G$ (which is isotopic to the meridian of $K$), and this is not a local relation. This seems to make calculations involving $K_{q,t}(K)$ more difficult. 
\end{remark}

\subsection{The unknot}\label{sec_topologicalunknot}
In this section we compute the $\SH_{q,t}$-module $K_{q,t}(K)$, where $K \subset S^3$ is the unknot. We first compute the bimodule structure of $\bar K_{q,t}(K)$. The key observation is the following lemma.

\begin{lemma}\label{lemma:hopf}
 Let $U \subset S^3$ be an open $\epsilon$-neighborhood of the union of the unknot $K$ and the graph $G$. Then the complement $S^3 \setminus U$ is diffeomorphic to $(T^2 \setminus p) \times [0,1]$. Under this identification, the meridian $m$ of $K$ is isotopic to the longitude $l_G$ of $G$, and the longitude $l$ of $K$ is isotopic to the meridian $m_G$ of $G$.
\end{lemma}
\begin{proof}
 If we write $U_a$ for $U$ with the arc joining $K$ and the loop in $G$ removed, then $U_a$ is the Hopf link, so $S^3\setminus U_a$ is the thickened torus $T^2 \times [0,1]$. Since the arc $a$ can be embedded in a vertical disc in $T^2 \times [0,1]$, it is isotopic to a vertical arc. Therefore, $S^3 \setminus U$ is diffeomorphic to $(T^2\setminus p)\times [0,1]$.
\end{proof}
\begin{corollary}\label{cor_unknotbimodule}
 We have an isomorphism $\bar K_{q,t}(K) \cong \SH_{q,t}$ of $\SH_{q,t}$-bimodules.
\end{corollary}
\begin{proof}
Let $M' = (S^3 \setminus K) \setminus (\textrm{nbhd of } G)$. By Lemma \ref{lemma:hopf}, there is a diffeomorphism $\varphi: M' \to (T^2\setminus p)\times [0,1]$. By Remark \ref{rmk:3manrephrasing}, $\bar K_{q,t}(K)$ is the quotient of $K_q(M')$ by the relation that says: if $L_p$ is a loop around the arc $A$ of the graph $G$, then $L_p = -q^2t^{-2}-q^{-2}t^2$. Since this is exactly the relation defining $K_{q,t}(T^2)$, this shows that $\varphi$ induces an isomorphism of vector spaces, which implies $\bar K_{q,t}(K) $ is isomorphic to $\SH_{q,t}$ as vector spaces. The fact that $\varphi$ respects the bimodule structure follows from the third sentence of Lemma \ref{lemma:hopf} and the discussion following Remark \ref{rmk:arcchoices}.
\end{proof}

\begin{theorem}\label{thm_topunknot}
 The left module $K_{q,t}(K)$ is isomorphic to $\e P^-$.
\end{theorem}
\begin{proof}
By Definition \ref{def_leftmodquotient}, $K_{q,t}(K) = \bar K_{q,t} \otimes_{\SH_{q,t}} \e P^-$. Then Corollary \ref{cor_unknotbimodule} shows that $K_{q,t}(K) = \SH_{q,t}\otimes_{\SH_{q,t}} \e P^-$, which completes the claim.
\end{proof}    

We write $\phi(P^-)$ for the twist of $P^-$ by the anti-automorphism $\phi$ defined in (\ref{eq_antiautomorphism}).
\begin{corollary}\label{corollary_pkp}
 If $K$ is the unknot, the vector space $ \phi(P^-)\e \otimes_{\SH_{q,t}} \bar K_{q,t}(K) \otimes_{\SH_{q,t}} \e P^-$ is 1-dimensional.
\end{corollary}
\begin{proof}
 By Corollary \ref{cor_unknotbimodule}, there is a bimodule isomorphism $\bar K_{q,t}(K) \cong \SH_{q,t}$. By Lemma \ref{signpairing}, the space of linear functions $\phi(P^-)\e \otimes_{\SH_{q,t}} \e P^- \to \C$ is 1-dimensional, which proves the claim.
\end{proof}

\begin{remark}\label{remark_generalpairing}
It is natural to ask whether this corollary can be generalized to non-trivial knots. This seems like a subtle question in general - in particular, the proof in the case of the unknot relies on the existence and (essential) uniqueness of Cherednik's pairing (see Lemma \ref{signpairing}), which is a non-trivial fact. 
\end{remark} 

\section{A topological interpretation of Cherednik's construction}\label{sec_topcherednik}
In Section \ref{sec_cabling} we described how the colored Jones polynomials of a cable of a knot $K$ can be computed from the skein module $K_q(S^3\setminus K)$. In particular, the colored Jones polynomials of torus knots can be computed from the skein module of the unknot. In this section we use the topological construction of the previous sections to define polynomials $J_{n,r,s}(q,t)$ that satisfy
\begin{equation}\label{toppolys}
 J_{n,r,s}(q,t=1) = J_{n}(K_{r,s};q)
\end{equation}
where $K_{r,s}$ is the $(r,s)$-torus knot. (Technically, $J_{n,r,s}$ is a rational function of $q$ and $t$, but this is because the Macdonald polynomials are rational functions of $q$ and $t$.)

In \cite{Che13}, Cherednik used the double affine Hecke algebra of type $\g$ to construct 2-variable polynomials that specialize to the colored Jones polynomials of type $\g$. We recall his construction for $\g = \sl_2$ and then show his polynomials are equal to the polynomials $J_{n,r,s}(q,t)$ (up to a renormalization).

\begin{rmk}
 It turns out that Cherednik's parameter $t$ is slightly different than ours. To try to eliminate confusion, in this section we will write $\tc \in \C^*$ for Cherednik's parameter and $t \in \C^*$ for ours. In general we will use a subscript $c$ to indicate objects which depend on Cherednik's parameter.
\end{rmk}

\subsection{Cherednik's construction}\label{sec_cheredniktorus}
We first recall Cherednik's construction from \cite{Che13}. Let $P^+_c = \C[X^{\pm 1}]$ be the polynomial representation of $\H_{q,\tc}$, and define the evaluation map 
\begin{equation*}
 \epsilon_c: P^+_c \to \C,\quad \epsilon_c(f(X)) = f(t_c)
\end{equation*}
In the notation of Section \ref{cor_pairings}, we can write this evaluation map as $\epsilon(-) = \langle 1, -\rangle_c$, where the pairing on the right is defined in Corollary \ref{cor_pairings}. 
If we identify $\e P^+_c \cong \C[x] = \C[X+X^{-1}]$ and restrict the evaluation map, for $g(x) \in \C[x]$ we obtain
\begin{equation}\label{cheredniksevaluation}
 \epsilon_c: \e P^+_c \to \C,\quad \epsilon_c(g(x)) = \epsilon_c(g(X+X^{-1})) = g(t_c+t^{-1}_c)
\end{equation}

We recall Cherednik's construction of an $\SL_2(\Z)$ action on $\HH_{q,\tc}$ via the following formulas:
\begin{alignat}{4}\label{cherednikssl2zaction}
 \left[\begin{array}{cc}1&1\\0&1\end{array}\right] &\mapsto \tau^+, &\quad \tau^+(X) &= X, &\quad \tau^+(T) &=T, &\quad \tau^+(Y) &= q^{-1}XY\\
 \left[\begin{array}{cc}1&0\\1&1\end{array}\right] &\mapsto \tau^-, &\quad \tau^-(X) &= qYX, &\quad \tau^-(T) &=T, &\quad \tau^-(Y) &= Y\notag
\end{alignat}

Let $\gamma_{r,s} \in \SL_2(\Z)$ be a matrix satisfying $\gamma_{r,s}(0,1) = (r,s)$, and let $p_n \in \C[X+X^{-1}] \subset \C[X^{\pm 1}]$ be the Macdonald polynomial of Theorem \ref{thdef}. Cherednik then gives the following definition:
\begin{definition}[\cite{Che13}]
The \emph{nonreduced DAHA-Jones invariant} is defined by
\begin{equation}\label{cherednikpoly}
 P_{n, r,s}(q,t_c) := (-q)^{rs(n^2-1)}\epsilon_c(\gamma_{r,s}(p_{n-1}(Y+Y^{-1}))\cdot 1)
\end{equation}
\end{definition}
These polynomials do not depend on the choice of $\gamma_{r,s}$ (see Lemma \ref{lemma_dependonrs}).

\begin{remark}\label{rmk_normalization}
 The polynomials above are defined in \cite[(2.18)]{Che13}, but we have changed the normalization by a power of $-q$. Note that his $b$ is our $n-1$, so $b(b+2) = n^2-1$. Also, Cherednik's $q^{1/2}$ is our $q^2$, and his $t^{1/2}$ is our $t$. Our Jones polynomials are normalized so that the $n^\mathrm{th}$ colored Jones polynomial of the unknot is $(-1)^{n-1}(q^{2n}-q^{-2n}) / (q^2-q^{-2})$.
\end{remark}

\subsection{The topological construction}
We now give a topological interpretation of these polynomials (after establishing some notation). Let $K$ be the unknot. Recall from (\ref{eq_symsigndecomp}) that as a $\C[x]$-module we have $K_{q,t}(K) = \C[x]\delta_t$, and that by Theorem \ref{thm_topunknot} the action of $\SH_{q,t}$ is determined by 
\begin{equation}\label{unknotaction}
 y\cdot \delta_t = -(tq^{-2}+t^{-1}q^2)\delta_t,\quad z\cdot \delta_t = -q^{-3}t^{-1}x\delta_t
\end{equation}
We first compare the module $K_{q,t}(K)$ to the classical skein module $K_q(K)$.
\begin{lemma}\label{lemma_teq1isskeinmodule}
 When $K$ is the unknot, the $A_q^{\Z_2}$-modules $K_q(K)$ and $K_{q,t=1}(K)$ are isomorphic.
\end{lemma}
\begin{proof}
 The formulas in (\ref{unknotaction}) follow directly from Theorem \ref{thm_topunknot}. As a $\C[x]$-module, the classical skein module $K_q(K)$ is freely generated by the empty link. If the formulas (\ref{unknotaction}) are specialized to $t=1$, then the action of $y$ and $z$ on $\delta_t$ agrees with the action of $y$ and $z$ on the empty link. Then Lemma \ref{lemma_invsubalgebraiso} shows that the $\C[x]$-module isomorphism sending $\delta_{t=1}$ to the empty link is an isomorphism of $A_q^{\Z_2}$-modules.
\end{proof}

We now need the analogue of the topological pairing in (\ref{knotpairing}) which is provided by Lemma \ref{signpairing}. We recall that this pairing is defined by
\begin{equation}
 \langle -,-\rangle_t: \phi(P^-)\e \otimes_{\SH_{q,t}} K_{q,t}(K) \to \C,\quad \langle \delta_t, \varnothing\rangle = 1
\end{equation}
(Recall that $\phi$ is the anti-automorphism defined in (\ref{eq_antiautomorphism}), and $\delta_t = 1\cdot (Y\e)$ is the generator of $\phi(P^-)\e$, and $\varnothing$ is the empty link in $K_{q,t}(K)$.) 
\begin{lemma}\label{lemma_teq1ispairing}
 When $t=1$, the pairing $\langle -,-\rangle_{t=1}$ is the same as the topological pairing from (\ref{knotpairing}).
\end{lemma}
\begin{proof}
The topological pairing is uniquely determined by the requirement $\langle \varnothing, \varnothing\rangle = 1$, and the algebraic pairing is uniquely determined by the requirement $\langle \delta_t,\delta_t\rangle_t = 1$. Furthermore, the isomorphism $K_{q,t}(K) \to \e P^-$ satisfies $\varnothing \mapsto \delta_t$, which completes the proof.
\end{proof}

To simplify notation, we will view this pairing as a functional on $K_{q,t}(K)$:
\begin{equation}\label{ourevaluation}
 \epsilon: K_{q,t}(K) \to \C,\quad\quad \epsilon(u) = \langle \delta_t, u\rangle_t
\end{equation}
\begin{lemma}
 We have the equality $\epsilon(f(x)) = f(-tq^{-2}-t^{-1}q^2)$.
\end{lemma}
\begin{proof}
 This follows from the identity $y\cdot \delta_t = -(tq^{-2}+t^{-1}q^2)\delta_t$ and the fact that $\phi(x) = y$.
\end{proof}

\begin{definition}
Let $p^-_n(x) \in \C[x]$ be the \emph{sign} Macdonald polynomials of Definition \ref{def_signmacpolys}, and let $\gamma_{r,s} \in \SL_2(\Z)$ satisfy $\gamma_{r,s}(0,1) = (r,s)$. We then define
\begin{equation}
 J_{n,r,s}(q,t) := (-q)^{rs(n^2-1)}\epsilon(\gamma_{r,s}(p^-_{n-1}(y))\cdot \varnothing)
\end{equation} 
\end{definition}
These polynomials do not depend on the choice of $\gamma_{r,s}$ by Lemma \ref{lemma_dependonrs}.

\begin{remark}
 Lemma \ref{lemma_teq1ispairing} shows that our pairing can be interpreted as the $t$-analogue of ``gluing the knot $K$ into $S^3\setminus K$ to obtain $S^3$.'' Under this interpretation, the polynomials $J_{n,r,s}(q,t)$ can be interpreted (roughly) as the evaluations of parallels of the $r,s$ torus knot embedded in $S^3\setminus (\mathrm{unknot } \cup G)$.
\end{remark}

\subsection{The comparison}\label{sec_comparison}
In this section we relate Cherednik's polynomials $P_{n,r,s}(q,t_c)$ to the polynomials $J_{n,r,s}(q,t)$ that were constructed in the previous section using the cabling formula. We also relate these polynomials to the colored Jones polynomials $J_n(K_{r,s}; q)$ of the $(r,s)$ torus knot.

\begin{theorem}\label{thm_cherednikpolyistoppoly_proved}
 We have the following equalities 
 \begin{align*}
P_{n,r,s}(q,t_c=-q^2t^{-1}) &= J_{n,r,s}(q,t)  \\
J_n(K_{r,s}; q) &= P_{n,r,s}(q,t_c=-q^2) = J_{n,r,s}(q, t=1)
 \end{align*}

\end{theorem}

\begin{remark}
 The equality $J_n(K_{r,s}; q) = P_{n,r,s}(q, t=-q^2)$ in Theorem \ref{thm_cherednikpolyistoppoly_proved} was proved in \cite[Thm. 2.8]{Che13} by direct computation without relying on skein modules or the cabling formula. We also remark that the key fact used to prove the first equality is an isomorphism of algebras $\SH_{q,t=-q^2} \cong \SH_{q,t=1}$, and that the standard polynomial representation at $t=-q^2$, when transferred along this isomorphism, becomes the sign representation. Then the second and third equalities are proved using the cabling formula, along with an identification of $\SH_{q,t=1} = A_q^{\Z_2} \cong K_q(T^2)$.
\end{remark}

 We begin with some lemmas which will be used in the proof. First, we recall the algebra $B'_q$, which is generated by elements $x,y,z$ subject to the relations
 \[
  [x,y]_q = (q^2-q^{-2})z,\quad [z,x]_q = (q^2-q^{-2})y,\quad [y,z]_q = (q^2-q^{-2})x
 \]
From Theorem \ref{sphericalrelations}, there is a surjection $f:B'_q \to \SH_{q,t}$ given by
\begin{equation}\label{mapfrombq}
 f(x) = (X+X^{-1})\e,\quad f(y) = (Y+Y^{-1})\e,\quad f(z) = q^{-1}(XYT^{-2}+X^{-1}Y^{-1})\e
\end{equation}

We first relate the $\SL_2(\Z)$ actions that are used in constructing both polynomials. The mapping class group $\SL_2(\Z)$ of $T^2 \setminus p$ acts on $K_q(T^2\setminus p)$, and we can transport this action to $B'_q$ using the isomorphism $B'_q = K_q(T^2\setminus p)$ of Theorem \ref{thm_bp}. We define two automorphisms $\tau_\pm :B'_q \to B'_q$:
\begin{alignat}{4}\label{sl2actiononB}
 \left[\begin{array}{cc}1&1\\0&1\end{array}\right] &\mapsto \tau^+, &\quad \tau^+(x) &= x, &\quad \tau^+(y) &=z, &\quad \tau^+(z) &= q^{-1}xz-q^{-2}y\\
 \left[\begin{array}{cc}1&0\\1&1\end{array}\right] &\mapsto \tau^-, &\quad \tau^-(x) &= z, &\quad \tau^-(y) &=y, &\quad \tau^-(z) &= q^{-1}zy - q^{-2}x\notag
\end{alignat}

 \begin{lemma}\label{lemma_sl2actiononB}
  The mapping class group $\SL_2(\Z)$ acts on $B'_q$ via formula (\ref{sl2actiononB}). The induced action of $\SL_2(\Z)$ on $\SH_{q,t}$ agrees with Cherednik's $\SL_2(\Z)$ action defined in (\ref{cherednikssl2zaction}).
 \end{lemma}
\begin{proof} These are all straightforward computations.
\end{proof}

We next relate the $\SH_{q,t_c}$-module $\e_c P_c^+$ used in Cherednik's construction to the module $K_{q,t}(K)\cong \e_t P^-$ used in the topological construction. We consider both these modules as modules over $B'_q$ using the surjection $f$ in (\ref{mapfrombq}). We will write $\e_c \in \HH_{q,t_c}$ and $\e_t \in \HH_{q,t}$ for the idempotent $(T+t^{-1})/(t+t^{-1})$. With this notation, we recall that we have the following equalities of $\C[x]$-modules:
\begin{equation}\label{isoppluspminus}
 \e_c P_c^+ = \C[x] 1_c \subset \C[X^{\pm 1}],\quad \e_t P^- = \C[x]\delta_t \subset \C[X^{\pm 1}]
\end{equation}
The modules $\e_c P_c^+$ and $\e_t P^-$ are isomorphic as $\C[x]$-modules via the isomorphism
\[
\varphi:\e_c P_c^+ \to \e_t P^-, \quad \quad \varphi(1_c) = \delta_t 
\]

\begin{lemma}\label{lemma_ppluspminusiso}
 If $t_c = -q^2t^{-1}$, then the map $g: \e_{c} P_c^+ \to \e_t P^-$ is an isomorphism of $B'_q$-modules. This map satisfies $\varphi(p_n(x)) = p^-_n(x)\delta_t$.
\end{lemma}
\begin{proof}
 A short computation shows that
 \[
  y\cdot 1_c = (t_c+t_c^{-1})1_c,\quad z\cdot 1_c = q^{-1}t^{-1}x1_c
 \]
Another short computation shows that
\[
 y\cdot \delta_t = -(tq^{-2}+t^{-1}q^2)\delta_t,\quad z\cdot \delta_t = -tq^{-3} x\delta_t
\]
If $t_c = -q^2t^{-1}$, then these formulas agree, and Lemma \ref{lemma_invsubalgebraiso} gives the first claim. The second claim follows from the fact that the Macdonald polynomials $p_n(x)$ and $p_n^-(x)$ are an eigenbasis for the operator $y$ for the modules $\e_c P_c^+$ and $\e_t P^-$, respectively.
\end{proof}

Finally, we relate the evaluation maps (\ref{cheredniksevaluation}) and (\ref{ourevaluation}) that are used to define the polynomials $P_{n,r,s}$ and $J_{n,r,s}$.
\begin{lemma}\label{lemma_compareevaluations}
 If $t_c = -q^2t^{-1}$, then for all $h(x) \in \e_c P^+_c$ we have
 \[
  \epsilon_c(h(x)) = \epsilon(\varphi(h(x))
 \]
\end{lemma}
\begin{proof}
 From (\ref{cheredniksevaluation}), we have $\epsilon_c(h(x)) = h(t_c+t_c^{-1})$. When $t_c = -q^2t^{-1}$ this agree with the equality $\epsilon(h(x)) = h(-tq^{-2}-t^{-1}q^2)$ from (\ref{ourevaluation}).
\end{proof}

We have now collected the facts needed to prove Theorem \ref{thm_cherednikpolyistoppoly_proved}. We first recall the equalities claimed in the theorem:
\begin{align*}
P_{n,r,s}(q,t_c = -q^2t^{-1}) &= J_{n,r,s}(q,t)  \\
J_n(K_{r,s}; q) &= P_{n,r,s}(q,t_c=-q^2) = J_{n,r,s}(q, t=1)
 \end{align*}
\begin{proof}(of Theorem \ref{thm_cherednikpolyistoppoly_proved})

We recall that the polynomials $P_{n,r,s}(q,t_c)$ and $J_{n,r,s}(q,t)$ are defined as follows:
\begin{align*}
 P_{n,r,s}(q,t_c) := (-q)^{rs(n^2-1)}\epsilon_c(\gamma_{r,s}(p_{n-1}(y))\cdot 1_c)\\
 J_{n,r,s}(q,t) := (-q)^{rs(n^2-1)}\epsilon(\gamma_{r,s}(p^-_{n-1}(y))\cdot \varnothing))
\end{align*}
The empty link is denoted $\varnothing$, and under the identification $K_{q,t}(K) = \e_t P^-$ we have $\varnothing \mapsto \delta_t$. By Lemma \ref{lemma_ppluspminusiso}, we see that $\varphi(y\cdot1_c) = y\cdot \varphi(1_c) = y\cdot \delta_t$. Then Lemma \ref{lemma_sl2actiononB} shows that $\varphi(\gamma_{r,s}(p_n(y))\cdot 1_c) = \gamma_{r,s}(p_n(y))\cdot \delta_t$. Finally, Lemma \ref{lemma_compareevaluations} completes the proof of the first equality.

To complete the proof of the theorem, we first note that when $t=1$, the module $K_{q,t=1}(K)$ and the pairing $\langle -,-\rangle_t$ are the same as the classical skein module $K_q(K)$ and the classical topological pairing. (These claims are Lemma \ref{lemma_teq1isskeinmodule} and \ref{lemma_teq1ispairing}, respectively.) The second equality in the statement of the theorem then follows from the cabling formula in Corollary \ref{corollary_topcablingformula}.
\end{proof}

\section{Iterated cablings of the unknot}\label{sec_iteratedcables}
The key facts that allowed the comparison of Cherednik's polynomials to colored Jones polynomials of torus knots are that the colored Jones polynomials satisfy a cabling formula and that torus knots are cables of the unknot. However, the cabling formula applies to all knots, and in particular can be used to write colored Jones polynomials of iterated cables of the unknot in terms of colored Jones polynomials of the unknot. In this section we use this observation to extend Cherednik's construction.

More precisely, let $\rr = (r_1,\ldots,r_m)$ and $\sss = (s_1,\ldots, s_m)$ with $r_i,s_i \in \Z$ relatively prime, and write $\rr_{k} = (r_1,\ldots,r_k)$ and similarly for $\sss_k$. 
\begin{definition}
Let $K^\tp(\rr_1,\sss_1)$ be the 0-framed $(r_1,s_1)$ torus knot, and let $K^\tp(\rr_m,\sss_m)$ be the $(r_m,s_m)$ topological cable of the knot $K^\tp(\rr_{m-1},\sss_{m-1})$. (See Definition \ref{def_topcable} - in particular, each $K^\tp(\rr_{k},\sss_k)$ is 0-framed.)
\end{definition}
Below we will define 2-variable polynomials $J_n(\rr,\sss; q,t) \in \C[q^{\pm 1},t^{\pm 1}]$, and we will show that they specialize to the colored Jones polynomials of the knot $K^\tp(\rr,\sss)$:
\begin{equation}
 J_n(\rr,\sss; q,t=-q^2) = J_n(K^\tp(\rr,\sss); q)
\end{equation}
When $m = 1$, this construction reproduces Cherednik's construction in \cite{Che13}. 

\begin{remark}
As we learned in the last section, to produce Jones polynomials using the DAHA we can either use the sign representation and specialize to $t=1$, or we can use the standard polynomial representation and specialize to $t=-q^2$. In this section we will make the latter choice (since for higher rank DAHAs this specialization is more natural).
\end{remark}

Before proceeding, we give an extension of Corollary \ref{corollary_topcablingformula} to iterated cables of the 0-framed unknot. 

\begin{remark} \label{remark_ueqy}
In Definition \ref{def_topneighborhoods} the identification of the skein modules of the neighborhoods $N_K$ and $N_{r,s}$ of a knot $K$ and its cable $K_{r,s}$ are the same since they are both 0-framed. We may therefore view the map $\Gamma_{r,s}^\tp$ of equation (\ref{def_gammatop}) as a map $\Gamma_{r,s}^\tp: \C[y] \to \C[y]$ by identifying $u = y$.
\end{remark}
\begin{corollary}\label{corollary_topiteratedcablingformula}
Given sequences $\rr$, $\sss$ as above, let $\Gamma^\tp_{\rr,\sss} = \Gamma^\tp_{r_1,s_1} \circ \cdots \circ \Gamma^\tp_{r_m,s_m}: \C[y] \to \C[y]$, where $\Gamma^\tp_{r,s}$ is defined in equation (\ref{def_gammatop}). Then we have the equality
\[
J_n(K^\tp(\rr,\sss)) = \langle \varnothing \cdot \Gamma^\tp_{\rr,\sss}(S_{n-1}(y)),\varnothing\rangle_{unknot}
\]
\end{corollary}
\begin{proof}
This follows from the proof of Corollary \ref{corollary_topcablingformula} combined with a simple induction.
\end{proof}

\subsection{Two variable polynomials for iterated cables}

Let $\phi:\H_{q,t} \to \H_{q,t}$ be the anti-automorphism from (\ref{eq_antiautomorphism}). Since $\phi(T) = T$, this anti-automorphism restricts to the spherical subalgebra $\SH_{q,t}$. In terms of the generators $x,y,z \in \SH_{q,t}$ from Theorem \ref{sphericalrelations}, we have
\[
 \phi(x) = y,\quad \phi(y) = x,\quad \phi(z) = z
\]
Let $P$ be the polynomial representation of $\H_{q,t}$ defined in Section \ref{dunkl} twisted by the anti-automorphism $\phi$, and let $P\e$ be its symmetrization. Explicitly, we may identify $P\e = \C[y]$, where the action of $\SH_{q,t}$ on $P\e$ is determined by the formulas
\begin{equation}
 f(y)\cdot y = f(y)y,\quad 1\cdot x = (t+t^{-1}),\quad 1\cdot z = q^{-1}t^{-1}y
\end{equation}
Recall that $\SL_2(\Z)$ acts on $\SH_{q,t}$ via (\ref{cherednikssl2zaction}). 
\begin{definition}
Given $r,s \in \Z$ relatively prime, let $\gamma_{r,s} \in \SL_2(\Z)$ be such that 
\[\gamma_{r,s}\left(\begin{array}{c} 0\\1\end{array}\right) = \left(\begin{array}{c} r\\s\end{array}\right)\] 
\end{definition}
\begin{lemma}\label{lemma_dependonrs}
 The element $\gamma_{r,s}(y) \in \SH_{q,t}$ depends only on $r,s$ and not on the choice of $\gamma_{r,s}$.
\end{lemma}
\begin{proof}
The stabilizer of $(0,1)^T$ in $\SL_2(Z)$ is the subgroup generated by $\tau^-$ (see formula (\ref{cherednikssl2zaction})). Since $\tau^-(Y) = Y$, we have $\tau^-(y) = y$, which proves the claim.
\end{proof}

We let $p_n(y) \in P\e$ be the \emph{standard} Macdonald polynomials defined in Theorem \ref{thdef}. To define our polynomials we will need the map $\iota_{q,t}$ of right $\C[Y+Y^{-1}]$-modules and the evaluation map $\epsilon_{q,t}:P\e \to \C$:
\begin{equation}\label{eq_iotaqt}
\iota_{q,t}: P\e \to \SH_{q,t},\quad \quad \iota_{q,t}(f(y)) = f(y), \quad \quad \epsilon_{q,t}(v) = \langle v \mid 1\rangle_{q,t}
\end{equation}
where the pairing defining $\epsilon_{q,t}$ is the one from Corollary \ref{cor_pairings}.

\begin{definition}\label{def_iteratedcable}
 Let $r,s \in \Z$ be relatively prime. Define a $\C$-linear map $\Gamma^\tp_{r,s;q,t}:P\e \to P\e$ by
\[
 \Gamma^\tp_{r,s;q,t}(p_{n-1}(y)) := (-q)^{rs(n^2-1)}1 \cdot \gamma_{r,s}(\iota_{q,t}(p_{n-1}(y)))
\]
Given sequences $\rr = (r_1,\cdots,r_m)$ and $\sss = (s_1,\cdots,s_m)$ with $r_i,s_i \in \Z$ relatively prime, define
\[
\Gamma^\tp_{\rr,\sss;q,t} := \Gamma^\tp_{r_1,s_1;q,t}\circ \cdots \circ \Gamma^\tp_{r_m,s_m;q,t},\quad \quad  J_n(\rr,\sss;q,t) := \epsilon_{q,t}(\Gamma^\tp_{\rr,\sss;q,t}(p_{n-1}(y)))
\]

\end{definition}

\begin{theorem}\label{thm_topiteratedcablespecialization}
We have the equality 
\[
 J_n(\rr,\sss;q,t=-q^2) = J_n(K(\rr,\sss); q)
\]
\end{theorem}
\begin{proof}
By Theorem \ref{sphericalrelations}, the algebra $\SH_{q,t=1}$ and $\SH_{q,t=-q^2}$ are isomorphic, and by Lemma \ref{lemma_sl2actiononB} this isomorphism is $\SL_2(\Z)$-equivarient. By Lemma \ref{lemma_ppluspminusiso} the right $K_q(T^2)$-module $K_q(N_K)$ is isomorphic to the right $\SH_{q,t=-q^2}$-module $P\e$. The inclusion maps $\iota_{q,t=-q^2}$ from (\ref{eq_iotaqt}) and $\iota$ from (\ref{eq_iotatop}) agree, and when $t=-q^2$ the Macdonald polynomials $p_n(y)$ specialize to the Chebyshev polynomials $S_n(y)$. Therefore, the linear map $\Gamma_{r,s;q,t=-q^2}^\tp$ is equal to the linear map $\Gamma_{r,s}^\tp$ of (\ref{def_gammatop}), which implies that $\Gamma_{\rr,\sss;q,t=-q^2}^\tp = \Gamma_{\rr,\sss}^\tp$. Finally, $\epsilon_{q,t=-q^2}(a) = \langle a,\varnothing\rangle_{unknot}$ by Lemmas \ref{lemma_teq1ispairing} and \ref{lemma_compareevaluations}. This means that when $t=-q^2$, the formula defining $J_n(\rr,\sss;q,t)$ agrees exactly with the cabling formula for $J_n(K^\tp(\rr,\sss))$ in Corollary \ref{corollary_topiteratedcablingformula}. This completes the proof of the theorem.
\end{proof}

\subsection{Examples}
In this section we include example calculations of the polynomials constructed in the previous section. We first consider the case of the knot $Kp := K(\rr,\sss)$ where $\rr = (2,2)$ and $\sss = (3,5)$, which is the $(2,5)$ cable of the $(2,3)$ knot (i.e. the trefoil). (We choose this example because it is the simplest example in which the ordering of the $\Gamma_{r_i,s_i}$ matters.) To produce this knot and calculate its (normalized) colored Jones polynomial in \texttt{Mathematica} using the \texttt{KnotTheory} package, one inputs

\begin{figure}
\includegraphics{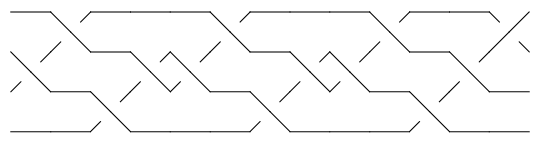}
\includegraphics{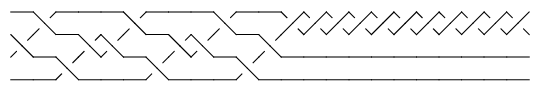}
\caption{The $(2,5)$ and $(2,-5)$ cables of the $(2,3)$ knot}\label{fig_25of23}
\end{figure}

\begin{equation}\label{eq_jones25of23}
\begin{array}{l}
\textbf{Kp} = \textbf{BR}[4,\{2,1,3,2,2,1,3,2,2,1,3,2,-1\}]\\
\textbf{ColouredJones}[\textbf{Kp},1][q^4](q^4-q^{-4})/(q^2-q^{-2})\\
q^{14}+q^{18}+q^{22}+q^{26}-q^{42}-q^{46}-q^{50}+q^{58}
\end{array}
\end{equation}
The output of this is the first image in Figure \ref{fig_25of23} along with the polynomial displayed in (\ref{eq_jones25of23}). Similarly, for the $(2,-5)$ cable $Km$ of the trefoil, the following produces the second image of Figure \ref{fig_25of23}:

\begin{equation}\label{eq_jones2m5of23}
\begin{array}{l}
\textbf{Km} = \textbf{BR}[4,\{ 2,1,3,2,2,1,3,2,2,1,3,2, -1,-1,-1,-1,-1,-1,-1,-1,-1,-1,-1\}]\\
\textbf{ColouredJones}[\textbf{Km},1][q^4](q^4-q^{-4})/(q^2-q^{-2})\\
-q^{-30}+q^{-6}+q^{-2}+q^2+q^6+q^{10}-q^{22}-q^{26}-q^{30}+q^{38}
\end{array}
\end{equation}

\begin{remark}
 In the definition of the $(r,s)$ cable of a knot $K$, one identifies the standard torus $S^1\times S^1$ with the boundary of a neighborhood $N_K$ of $K$. This identification is done using Lemma \ref{lemma_meridianlongitude} - in other words, the longitude of the boundary torus of $N_K$ must have $0$ linking number with the knot $K$. This explains why the first image in Figure \ref{fig_25of23} has $1$ `negative' crossing at the right, instead of $5$ `positive' crossings (since $6$ negative crossings must be added to correct for the framing/linking). Similarly, the second image of Figure \ref{fig_25of23} has $11$ negative crossings instead of $5$.
\end{remark}

We have also implemented code to compute the polynomials of Definition \ref{def_iteratedcable} in \texttt{Mathematica} using the \texttt{NCAlgebra} package. Running this code produces the following polynomials:

\begin{eqnarray*}
  (1 - q^4 t^2)J_2(Kp; q,t) &=& 
  q^{32} \left(\frac{1}{t}-t^3\right)+
  q^{44} \left(-\frac{1}{t^{15}}-\frac{1}{t^{13}}\right)+
  q^{48} \left(\frac{1}{t^{11}}+\frac{1}{t^9}\right)+
  q^{52} \left(-\frac{1}{t^{13}}+\frac{1}{t^9}\right)+\\
  &{ }& q^{56} \left(-\frac{1}{t^{13}}+\frac{2}{t^9}-\frac{1}{t^5}\right)+
  q^{60} \left(-\frac{1}{t^{11}}+\frac{1}{t^9}+\frac{1}{t^7}-\frac{1}{t^5}\right)+\\
  &{ }& q^{64} \left(\frac{1}{t^9}+\frac{1}{t^7}-\frac{1}{t^5}-\frac{1}{t^3}\right)+
  q^{68} \left(-\frac{1}{t^5}+\frac{1}{t}\right)
\\ 
  (1 - q^4 t^2)J_2(Km; q,t) &=& 
  q^{-28}\left(\frac{1}{t}-t^3\right) +
  q^4 \left(-\frac{1}{t^5}-\frac{1}{t^3}\right)+
  q^8 \left(\frac{1}{t}+t\right)+
  q^{12} \left(-\frac{1}{t^3}+t\right)+\\
  &{ }& q^{16} \left(-\frac{1}{t^3}+2 t-t^5\right)+
  q^{20} \left(-\frac{1}{t}+t+t^3-t^5\right)+\\
  &{ }& q^{24} \left(t+t^3-t^5-t^7\right)+
  q^{28} \left(-t^5+t^9\right)
\end{eqnarray*}

For completeness, we also include the following polynomials for the trefoil $K_{2,3}$:
\begin{eqnarray*}
 J_2(K_{2,3};q,t) &=& q^{12} \left(\frac{1}{t^5}+\frac{1}{t^3}\right)+q^{16} \left(\frac{1}{t^3}-t\right)\\
 (1 - q^4 t^2)J_3(K_{2,3};q,t) &=& q^{24} \left(-\frac{1}{t^{10}}-\frac{1}{t^8}\right)+q^{32} \left(-\frac{1}{t^8}+\frac{1}{t^4}\right)+q^{28} \left(\frac{1}{t^6}+\frac{1}{t^4}\right)+\\
& { } & q^{36}
\left(-1-\frac{1}{t^8}+\frac{2}{t^4}\right)+
+q^{40} \left(-1-\frac{1}{t^6}+\frac{1}{t^4}+\frac{1}{t^2}\right)+\\
&{ }& q^{44} \left(-1+\frac{1}{t^4}+\frac{1}{t^2}-t^2\right)+q^{48}
\left(-1+t^4\right)
\end{eqnarray*}

\appendix
\section{Comparing polynomials for iterated cables}\label{sec_Citeratedcables}

After the first version of the present paper appeared, Cherednik and Danilenko gave a construction of certain polynomials for general $\g$ in \cite{CD14}. 
For $\g = \sl_2$, the constructions in \cite{CD14} and in the present paper are similar but not quite identical. It turns out that in the specialization $t=-q^2$, the differences can be explained by two different cabling procedures.
Our construction for iterated torus knots involved the \emph{topological pairs} $(\rr,\aaa)$ and uses the standard topological cabling $K^\tp_{r,a}$ of a knot $K$  (we refer to $K^\tp_{r,a}$ as the \emph{topological cable} of $K$). However, the definition in \cite{CD14} uses the \emph{Newton pairs} $(\rr, \sss)$, which are related to the topological pairs via equation (\ref{eq_newton}).
In Section \ref{sec_cabling} we described a slightly non-standard procedure for 
constructing a cable $K^\alg_{r,s}$ of a framed knot $K$, which we call the 
\emph{algebraic cable} of $K$. Iterating these procedures (starting with the $0$-
framed unknot), we obtain two framed knots: $K^\tp_{\rr,\aaa}$ and $K^\alg_{\rr,
\sss}$. It turns out that these two knots are the same, up to an overall framing, and this implies their colored Jones polynomials are equal up to multiplication by a power of $q$. 

\begin{remark}\label{rmk_compareforallt}
At this point we will not describe a precise relationship between the constructions for arbitrary $t$. However, it seems likely that the monomial in $q$ in Definition \ref{def_iteratedcable} can be modified (to a rational function involving $t$ and $q$) in a way that would reproduce the polynomials defined in \cite{CD14}. In particular, the constants appearing in our construction are (powers of) eigenvalues of the Dehn twist of the solid torus, which has Chebyshev polynomials as eigenvectors. This Dehn twist acts on the torus and on the solid torus, and it has a $t$-deformation which is the subgroup of lower-triangular matrices in $\SL_2(\Z)$ acting on the DAHA and its polynomial representation. This operator on the polynomial representation is diagonalized by Macdonald polynomials, with eigenvalues depending on $q$ and $t$.
\end{remark}

We first recall their construction when it is specialized to $\g = \sl_2$ after  fixing some notation. Let $\rr = (r_1,\ldots,r_m)$ and $\sss = (s_1,\ldots,s_m)$ with $r_i, s_i \in \Z$ relatively prime be \emph{Newton pairs}, and let $\aaa = (a_1,\ldots,a_m)$ be determined by 
\begin{equation}\label{eq_newton}
 a_1 = s_1, \quad a_{i+1} = s_{i+1} + r_i r_{i+1} a_i
\end{equation}

\begin{remark}
 The terminology `Newton pairs' is used because they are the numbers produced by applying Newton's method to produce a multiplicative expansion of a Puiseux series describing an irreducible singularity of an algebraic planar curve. (See \cite[Appendix 1A]{EN85}.)
\end{remark}

\begin{definition}
Given $\rr,\sss$ we write $K^\alg(\rr,\sss)$ for the iterated algebraic cable of the $0$-framed unknot. 
\end{definition}

The algebraic cabling formula of Corollary \ref{corollary_algcablingformula} has a simple extension to iterated cables, where in the definition of $\Gamma_{r,s}^\alg$ in (\ref{def_gammaalg}) we identify $u=y$ as in Remark \ref{remark_ueqy}.
\begin{corollary}\label{corollary_algiteratedcablingformula}
Given sequences $\rr$, $\sss$ as above, let $\Gamma^\alg_{\rr,\sss} = \Gamma^\alg_{r_1,s_1} \circ \cdots \circ \Gamma^\alg_{r_m,s_m}: \C[y] \to \C[y]$, where $\Gamma^\alg_{r,s}$ is defined in equation (\ref{def_gammaalg}). Then we have the equality
\[
J_n(K^\alg(\rr,\sss)) = q^\bullet \langle \varnothing \cdot \Gamma^\alg_{\rr,\sss}(S_{n-1}(y)),\varnothing\rangle_{unknot}
\]
where the exponent in $q^\bullet$ depends on $\rr$, $\sss$, and $n$.
\end{corollary}
\begin{proof}
This follows from the proof of Corollary \ref{corollary_algcablingformula} combined with a simple induction.
\end{proof}

\begin{lemma}[{\cite[Proposition 1A.1]{EN85}}]\label{lemma_en}
 Let $a_i$ be as defined in (\ref{eq_newton}). Then the following iterated cables produce the same knot up to an overall framing twist:
 \[
  K^{\tp}(\rr,\aaa),\quad K^\alg(\rr,\sss)
 \]
\end{lemma}

We now recall the definition of the polynomials $JD_{\rr,\sss}(n; q,t)$ given in \cite[Thm. 2.1]{CD14} for $\g = \sl_2$. We will use the notation of Section \ref{sec_cheredniktorus}. Given $\rr,\sss$ as above, let $\gamma_i \in \SL_2(\Z)$ be such that
\begin{equation*}
 \gamma_i\left(\begin{array}{c}0 \\ 1\end{array}\right) = \left(\begin{array}{c}r_i\\s_i\end{array}\right)
\end{equation*}
Write $\Gamma_i^c: P\e_c \to P\e_c$ for the $\C$-linear map given by $\Gamma_i^c(f(Y+Y^{-1})) = 1\cdot \gamma_i(f(Y+Y^{-1}))$, and write $\Gamma^c_{\rr,\sss} = \Gamma^c_1\circ\cdots\circ \Gamma^c_m$.

\begin{definition}\label{def_cditerated}
Given $\rr,\sss$ as above, \cite{CD14} defines
\begin{equation}
JD_{n, \rr,\sss}(q,t) := \epsilon_c(\Gamma^c_{\rr,\sss}(p_{n-1}(Y+Y^{-1})))
\end{equation}
\end{definition}
\begin{remark}
This definition has been rescaled by a constant so that it will specialize to our convention for the Jones polynomials. We have also used the fact that the pairing $\langle-,-\rangle_c: P\e_c \otimes \e_c P \to \C$ is symmetric.
\end{remark}
\begin{theorem}
 We have the equality 
 \[ 
 JD_{n,\rr,\sss}(q,t=-q^2) = q^k J_n(K^\alg(\rr,\sss)) = q^j J_n(K^\tp(\rr,\aaa))
 \]
for some integers $j,k$ that depend on $\rr$, $\sss$, and $n$.
\end{theorem}
\begin{proof}
 We first remark that if we place a Jones-Wenzl idempotent on two different framings of a knot $K$ and then evaluate in $S^3$, the results will differ by a power of $-q$. (This is explained in the proof of Corollary \ref{corollary_topcablingformula}.) This combined with Lemma \ref{lemma_en} shows the second equality. The arguments of Theorem \ref{thm_topiteratedcablespecialization} show that when $t=-q^2$, the formula defining $JD(\rr,\sss;q,t=-q^2)$ reproduces the \emph{algebraic} cabling formula for iterated torus knots given in Corollary \ref{corollary_algiteratedcablingformula}, which proves the first equality.
\end{proof}

\bibliography{link_bibtex_july17_toppaper_2014}{}
\bibliographystyle{amsalpha}

\end{document}